\documentclass[12pt]{article}

\usepackage{amsmath,amsthm,amssymb}
\usepackage{amsthm,amsfonts}
\usepackage[english]{babel}
\usepackage{amsfonts}
\usepackage{amssymb}
\usepackage{amstext}

\usepackage{amsfonts, amstext,amscd,amstext,latexsym}

\begin{document}

\newtheorem{example}{Example}

\newtheorem{theorem}{Theorem} \newtheorem{corollary}{Corollary}
\newtheorem{lemma}{Lemma} \newtheorem{proposition}{Proposition}
\newtheorem{definition}{Definition}

\newcommand{\z}{\mathbb{Z}} \newcommand{\re}{\mathbb{R}}
\newcommand{\tn}{\mathbb{T}^N}
\newcommand{\rn}{\mathbb{R}^N}
\newcommand{\espaco}{[0,1]^{\mathbb{N}}}
\newcommand{\supp}{\mbox{supp}}
\newcommand{\nat}{\mathbb{N}}
\newcommand{\Rm}{{\noindent \sc Note. \ }}

\def\cqd {\,  \begin{footnotesize}$\square$\end{footnotesize}}
\def\cqdt {\hspace{5.6in} \begin{footnotesize}$\blacktriangleleft$\end{footnotesize}}
\def\refname{References}
\def\bibname{References}
\def\chaptername{\empty}
\def\figurename{Fig.}
\def\abstractname{Abstract}

\title{An introduction to Coupling}
% Use \titlerunning{Short Title} for an abbreviated version of
% your contribution title if the original one is too long
\author{Artur O. Lopes}
% Use \authorrunning{Short Title} for an abbreviated version of
% your contribution title if the original one is too long \email{alopes@mat.ufrgs.br}}
%
% Use the package "url.sty" to avoid
% problems with special characters
% used in your e-mail or web address
%
\maketitle

\bigskip
\abstract{In this review paper, we describe the use of couplings in several different mathematical problems.
We consider the total variation norm, maximal coupling, and the $\bar{d}$-distance.
We present a detailed proof of a result recently proved: the dual of the Ruelle operator is a contraction with respect to  $1$-Wasserstein distance.
We also show exponential convergence to equilibrium in the state space for finite-state Markov chains when the transition matrix $\mathcal{P}$ has all entries positive.}

In this new  version, we describe in more detail the line of reasoning followed in the work previously published as a chapter in 
``Modeling, Dynamics, Optimization and Bioeconomics II'', Springer Verlag (2017).

\section{Introduction}

This is a review paper presenting some examples that were described in the literature where couplings are used  for deriving results in Ergodic Theory and Probability. We present several simple calculations that we believe can help the beginner  in the area. We are writing for a broad audience and not for the expert.

Our purpose is to present such results in a more direct approach. This will avoid the reader who is interested in the topic from having to look in several different references.

We describe the results in a language that is more familiar to the Dynamical Systems audience. The proofs will be aligned with the point of view of Measure Theory.

In  section \ref{Ru1}, we present some definitions and mention some related results which are required in section \ref{Ru} where we consider the dual of the Ruelle operator.

In section \ref{Max}, we consider the total variation norm and the maximal coupling (see Theorem \ref{uu}). 

In section \ref{pri}, we consider stopping and random  times and their relation to   some special couplings.
Our main result in this section is to show the next result.

Suppose $ \mathcal{P}$ is a d by d line stochastic matrix with entries $P_{ij}>0$, $i,j=1,2,...,d.$ Denote by $0<\rho$ the infimum of the values $P_{ij}>0$, $i,j=1,2,...,d.$
Given an initial vector of probability $\nu= (\nu_1,\nu_2,...,\nu_d)$ one can define a Markov probability $P_\nu$ on $\{1,2,...,d\}^\mathbb{N}.$ This Markov probability describes the probabilities of the associated Markov Stochastic Process $Y_n^\nu$, $n \in \mathbb{N}.$

 If the initial stationary vector of probability $\lambda$ satisfies 
 $\lambda \, \mathcal{P}=\lambda$, then the associated Markov probability $P_\lambda$ is stationary, which means that $P_\lambda$ is invariant for the shift $\sigma:\{1,2,...,d\}^\mathbb{N}\to \{1,2,...,d\}^\mathbb{N}$.
 
 Note that for all $n \in \mathbb{N}$ we get that $P_\lambda (X^\lambda_n = j)= \lambda_j.$

\begin{theorem} \label{pco}
Denote by $Y_n^\nu$, $n \in \mathbb{N}$, the Markov probability associated with the stochastic matrix $\mathcal{P}$, with initial vector of probability $\nu$.
Then, for any $n$ we have
$$ | \lambda - P_\nu (Y_n^\nu \in .) |_{tv} \leq \,2\, (1-\rho)^n,$$
and this describes the speed of convergence to the equilibrium $\lambda$, when time goes to infinity, for the  Markov Process $Y_n^\nu$ .
\end{theorem}

In   sections \ref{eeuu} and \ref{sec},
the proof of the exponential convergence to equilibrium will be completed (showing Theorem \ref{pco}).

In section \ref{dbar}, we consider the $d$-bar distance among Bernoulli probabilities (see Theorem \ref{CC}).

In Section \ref{Ru}, we outline the proof  that  the dual of the Ruelle operator is a contraction for the $1$-Wasserstein distance (definitions in Section \ref{Ru1}); it is basically the same as the one presented  in \cite{KLS},  but with some shortcuts.

The use of coupling and the Wasserstein distance can be useful for estimating decay of correlations (see \cite{BFG2}  \cite{GP} and \cite{Sulku})
and also in other different dynamical and ergodic  problems (see \cite{Kl} and \cite{Aus}).

We believe it is important to describe interesting plans or couplings that can be used in estimations of different natures. Variations of these plans can be helpful to solve  other open problems.
\medskip

According to Frank den Hollander \,\cite{Hol}: {\it coupling is an art, not a recipe}.
\medskip

We would like to thank Anthony Quas,  Diogo Gomes, Adriana Neumann, and Rafael Souza for helpful conversations during the period we were writing this review paper.

\bigskip

\section{Some definitions and the dual of the Ruelle Operator} \label{Ru1}

\begin{definition}
Given the Bernoulli space
$\Omega=\{1,2,...,d\}^\mathbb{N}$,
each probability  $\Gamma$ in $\Omega \times \Omega$ is called a plan.

Given the Bernoulli space
$\Omega=\{1,2,...,d\}^\mathbb{N}$ and two probabilities $\mu$ and $\nu$ on the natural Borel sigma algebra of $\Omega$, a coupling of $\mu$ and $\nu$ is a plan $\Gamma$ on the product space $\Omega \times \Omega$, such that, the first marginal of $\Gamma$ is $\mu$ and the second is $\nu$.
\end{definition}

\medskip

$\mathcal{C} (\mu,\nu)$ by definition is the set of plans $\Gamma$ on $\Omega \times \Omega$ such that the projection in the first coordinate is $\mu$ and in the second is $\nu$. A {\bf  very } particular example is the product probability $\mu \times \nu$.

\begin{definition}
Given a distance $d$ on $\Omega$
we denote
\begin{equation} \label{khl} W_1(\mu,\nu) = \inf_{\Gamma \in \mathcal{C}(\mu,\nu)}  \, \int d(x,y) d\, \Gamma(dx,dy).
\end{equation}

The above  expression defines a metric on the space of probabilities over $\Omega$ which is compatible with the weak$*$-convergence (see \cite{Vi1}). This value is called the $1$-Wasserstein distance of $\mu$ and $\nu$.
\end{definition}

Each plan (it can exist more than one) that realizes the above infimum is called an optimal plan for $d$.

 We denote by $d_1(\mu,\nu)= W_1(\mu,\nu)$ the corresponding metric in the set of probabilities on $\Omega$.

\begin{definition}
More generally, given a continuous function $c:\Omega \times \Omega \to \mathbb{R}$ and fixed $\mu$ and $\nu$, one can be interested in
$$  \inf_{\Gamma \in \mathcal{C}(\mu,\nu)}  \, \int c(x,y) d\, \Gamma(dx,dy).$$

Each plan $\Gamma$ (it can be more than one) that realizes the above infimum is called an optimal plan, or optimal coupling, for $c$ and the pair $\mu$ and $\nu$.
\end{definition}

\medskip

In general is not so easy to identify exactly the optimal plan $\Gamma$. Anyway, if we are lucky enough to find a plan that is almost optimal, then we can get some interesting results. In simple words, this is the main issue on Coupling Theory.

The Kantorovich duality theorem (see \cite{Vi1}) is a main result that also helps to get good estimates in problems of a different nature:
\begin{equation} \label{falt1} d_1( \mu_1,\mu_2)\, = \sup_{\phi:X\to \mathbb{R}\,\,\text{has $d$\, Lipchitz constant}\, \leq 1} \,\{\, \int \phi \, d\,\mu_1 - \int \phi\, d\mu_2\,\}.\end{equation}

If $\mu_n \to \mu$ in weak*-convergence, and one is interested in the speed of convergence to zero  of $d_1(\mu_n,\mu)$, we point out that the equivalent expressions \eqref{khl} and \eqref{falt1} can provide  upper and lower bounds for each $n$ by taking, respectively,  {\it smart} choices on $\Gamma_n$ or $\phi_n$.

The total variation norm (to be defined later) is   a different way to measure the distance between two probabilities. It is not equivalent to weak-$*$-convergence. This will also be considered  in the text.

\medskip

One of our purposes is to illustrate through several examples the use of coupling in interesting problems.

\bigskip

Suppose $\theta<1$ is fixed.
On the Bernoulli space $\Omega=\{1,2,...,d\}^\mathbb{N}$ we consider the metric  $d_\theta$.
By definition $ d_\theta(x,y)= \theta^N $ where $x_1=y_1,..,x_{N-1}=y_{N-1}$ and  $x_{N}\neq y_{N}.$

We briefly mention some properties related to Gibbs states of
Lipchitz potentials (see \cite{PP} for general results).

\begin{definition}
Given $A: \Omega \to \mathbb{R},$
the Ruelle operator $\mathcal{L}_A$  acts on functions $\psi:\Omega \to \mathbb{R} $ in the following way
$$\varphi(x)=\mathcal{L}_A(\psi)(x)= \sum_{a=1}^d e^{A(ax)}\psi(ax).$$

By this we mean $\mathcal{L}_A(\psi)=\varphi.$
\end{definition}

Suppose $A=\log J$ is Lipchitz and normalized, that is, for any $x \in \Omega$ we have $\mathcal{L}_{\log J} (1) (x) =1.$
\bigskip

All probabilities we consider will be over the Borel sigma-algebra $\mathcal{B}$ of $\Omega$.

\begin{definition}
Given the continuous potential $\log J:\Omega\to\mathbb{R}$
let $\mathcal{L}_{\log J}^{*}$ be the operator
on the set of Borel finite measures on $\Omega$ defined so that
$\mathcal{L}_{\log J}^{*}(\nu)$, for each Borel measure $\nu$, satisfies:
 \[
  \int_{\Omega} \psi \ d\mathcal{L}^{*}_{\log J}(\nu)
  =
  \int_{\Omega} \mathcal{L}_{\log J}(\psi)\, d\nu,
 \]
for all continuous functions $\psi$.
\end{definition}

Suppose that $\mathcal{L}_{\log J}(1)=1$, then
$\mathcal{L}^{*}_{\log J}$ takes probabilities in probabilities.
A probability which is fixed for such an operator  $\mathcal{L}^{*}_{\log J}$ is invariant for the shift and called a $g$-measure.
In the case $\log J$ is Holder, such a fixed point is unique and is the equilibrium (Gibbs) state for $\log J$ (see \cite{PP}). We will denoted it by $\mu_{\log J}.$
We point out that given any probability $\rho$ on $\Omega$, we get that weakly* $(\mathcal{L}^{*}_{\log J})^n (\rho) \to \mu_{\log J}$ (see for instance \cite{PP} or \cite{LFT}). Note, however, that for a probability $\rho$ on $\Omega$ which is $\sigma$-invariant, the image  $(\mathcal{L}^{*}_{\log J})^n (\rho)$ is not necessarily $\sigma$-invariant (see \cite{LR}).

Is it true that there exists a metric $d$ equivalent to $d_\theta$ such that $\mathcal{L}_{\log J}^*$ is a contraction in the $1$ Wassertein distance $d_1$ associated to such $d$? The answer to this question is yes (see \cite{KLS}), and we will address the question in the section \ref{Ru}. Before that, we will present some more basic material in the next sections.

M. Stadlbauer presented a proof with an affirmative answer to the above question in a more general setting.

%For continuous time transport on the space of Gibbs probabilities we refer the reader to \cite{KGLM}.

\bigskip

\section{The total variation norm and the maximal coupling} \label{Max}

Suppose $\rho$ is a signed measure in the Borel sigma-algebra $\mathcal{B}$ of $\Omega=\{1,2,..,d\}^\mathbb{N}$ such that $\rho(\Omega)=0$.

\begin{definition}
The total variation of a signed measure $\rho$ as above is, by definition
$$ |\rho|_{tv} =  2\, \sup_{A \in\mathcal{B}} \rho (A).$$
\end{definition}

One can show by duality (see \cite{Vi1}) that
$$ |\rho|_{tv}= \sup_{|\phi|_\infty\leq 1} \int \phi\, d \rho .$$
where $\phi$ are measurable and bounded and $|\phi|_\infty $ is the supremum norm.

Given two probabilites $\mu_1$ and $\mu_2$ in $\Omega$ one can consider the distance
$|\mu_1-\mu_2|_{tv},$ called the total variation distance of $\mu_1$ and $\mu_2$. This is a different way (compared with \eqref{khl}) to measure the distance among two probabilities.

This distance is also known as the strong distance, in opposition to the more well known  concept of weak convergence of probabilities.

Remember that we denote $\mathcal{C} (\mu_1,\mu_2) $ the set plans $\Gamma$ in $\Omega\times \Omega$ such that the projection in the first coordinate is $\mu_1$ and in the second is $\mu_2$.
\medskip

\begin{proposition}
{\bf :} Given $\Gamma \in \mathcal{C} (\mu_1,\mu_2) $, then,
$$ |\mu_1-\mu_2|_{tv}\le 2 \,\,\Gamma\, \{(x,y)\,|\,  x \neq y \} .$$
\end{proposition}

{\bf Proof:} Given $A \in \mathcal{B}$ we have that
$$\mu_1(A)-\mu_2(A) = \int \int_{x\in A}\, \Gamma(dx,dy) -  \int_{y\in A} \int\, \Gamma(dx,dy)=$$
$$[\, \int_{x=y} \int_{x \in A}\, \Gamma(dx,dy) + \int_{x\neq y} \int_{x \in A}\, \Gamma(dx,dy)\,] -$$
 $$[\, \int_{y\in A} \int_{x= y}\, \Gamma(dx,dy)+
\int_{y \in A} \int_{x\neq y}\, \Gamma(dx,dy)\,]=$$
$$\int_{x\neq y} \int_{x \in A}\, \Gamma(dx,dy)- \int_{y \in A} \int_{x\neq y}\, \Gamma(dx,dy)\leq \int_{x\neq y} \int_{x \in A}\, \Gamma(dx,dy).  $$

Remember that $$ |\mu_1-\mu_2|_{tv} =  2\, \sup_{A \in\mathcal{B}} (\mu_1 - \mu_ 2) (A).$$

Taking supremum in $A \in \mathcal{B}$ we get the claim. \qed

We follow the general reasoning of \cite{Hol} and \cite{Lin}.

Suppose $\mu_1$ and $\mu_2$ are two different probabilities on the Bernoulli space $\Omega$.

Put $\lambda= \mu_1 + \mu_2,$ and
$$g =\frac{d \mu_1}{d \lambda},\, g' =\frac{d \mu_2}{d \lambda}.$$

Note that $g,g^\prime \leq 1$, because for nontrivial  Borel set $B$ we have that $\lambda (B)\geq\mu_1(B) $ and $\lambda (B)\geq\mu_2(B) $. There exist sets of positive measure  (for $\mu_1$ and $\mu_2$) where the inequality is strict.

Now we define $Q$ on $\Omega$ by
\begin{equation} \label{isto} d\, Q\, = \,(\,g \wedge \,g '\,) \,d \lambda ,\end{equation} where  $\,g \wedge \,g '$ denotes the infimum of $g$ and $g'$.

In this way, if $\lambda(B)=0$ we get that $Q(B)=0.$

Note that by Kantorovich Duality (see \cite{Vi1})
$$ | \mu_1-\mu_2|_{tv} = \sup_{|f|\leq 1,\, f \,\,\text{measurable}\,} | \int f\, d\mu_1 -  \int f\, d\mu_2\,|.$$

Therefore,
$$ | \mu_1-\mu_2|_{tv} = \sup_{|f|\leq 1,\, f \,\,\text{measurable}\,} | \int f\, (g- g') d \lambda\,|=$$
$$ \int_{g \geq g'}  1\, (g- g') d \lambda\,+ \int_{g < g'}  (-1)\, (g- g') d \lambda\, =\int  |\,g- g'\,| d \lambda=$$
\begin{equation} \label{pre1} \int (g -  \,g \wedge \,g '\,)\,d \lambda\,+ \int (g ' -  \,g \wedge \,g '\,)\,d \lambda\, = 2 (1- \int  \,(g \wedge \,g ')\, d \lambda ).
\end{equation}

This shows that 
\begin{equation} \label{qre1}\int  \,(g \wedge \,g ')\, d \lambda<1.
\end{equation}

Consider $\varphi: \Omega \to \Omega \times \Omega $ by $\varphi(x)=(x,x).$

Finally, we denote by $\hat{Q}= \varphi^* (Q).$

Note that the support of $\hat{Q}$ is the diagonal $\Delta$ in $\Omega \times \Omega.$

Now, let $\gamma= \hat{Q} (\Delta)=Q(\Omega)$.

Note that from \eqref{isto} and  \eqref{qre1} we get that $\gamma<1$.

We call $\nu_1= \mu_1 - Q$ and $\nu_2 = \mu_2 - Q$, and finally we define a plan $\pi$ in $\Omega \times \Omega$
by
$$ \pi = \hat{Q} \,+\, \frac{\nu_1 \otimes \nu_2 }{1- \gamma}.$$

This plan is sometimes called {\bf maximal coupling}.
\bigskip

Note that $\nu_2 (\Omega)= \mu_2 (\Omega) - Q(\Omega) = 1 - \gamma= \nu_1 (\Omega)= \mu_1 (\Omega) - Q(\Omega) $.

We claim that $\pi$ projects in the first coordinate on $\mu_1$. Indeed,
$$ \pi (A \times \Omega) = \hat{Q} (A \times \Omega) \,+\, \frac{\nu_1 \otimes \nu_2 }{1- \gamma} (A \times \Omega)=$$
$$   \hat{Q} (A \times A) \,+\, \frac{\nu_1 (A) \otimes \nu_2 (\Omega) }{1- \gamma}= $$
$$Q(A) + \frac{\nu_1 (A)\, (1- \gamma)}{1- \gamma} =  Q(A) + (\mu_1 (A) - Q(A))= \mu_1(A).      $$

The above also shows that $\pi$ is a probability.

Moreover,
$\pi$ projects in the second coordinate on $\mu_2$. Indeed,
$$ \pi (\Omega \times A) = \hat{Q} (\Omega \times A) \,+\, \frac{\nu_1 \otimes \nu_2 }{1- \gamma} (\Omega \times A)=$$
$$   \hat{Q} (A \times A) \,+\, \frac{\nu_1 (\Omega) \otimes \nu_2 (A) }{1- \gamma}=
   $$
   $$Q(A) + \frac{\nu_2 (A)\, (1- \gamma)}{1- \gamma} =  Q(A) + (\mu_2 (A) - Q(A))= \mu_2(A).      $$

\bigskip
In this way $\pi \in \mathcal{C} (\mu_1,\mu_2) $.
Therefore, it follows from a previous result  that for such plan it is true the property $|\mu_1-\mu_2|_{tv}\leq  2 \,\,\pi \,\{(x,y)\,|\,  x \neq y \}.$
\medskip
Now we will show:

\medskip
\begin{theorem} \label{uu} The plan $\pi$ defined above satisfies
$$ |\mu_1-\mu_2|_{tv}=  2 \,\,\pi \,\{(x,y)\,|\,  x \neq y \} .$$
\end{theorem}

{\bf Proof:}
First note that as $|g- g'|= g + g' - 2 \,(g \, \wedge g')$, we have by (\ref{pre1})
$$ |\mu_1-\mu_2|_{tv}=\int |g\,-g' |\, d \,\lambda  = 2 \,[\,1 -\int\,(\, g\,\wedge g '\,)\,d\, \lambda\,]=$$
 $$2\, (1- Q(\Omega))=2\,(1-\gamma)\,\geq \,2\, \pi(\Delta^c)=  2 \,\,\pi \,\{(x,y)\,|\,  x \neq y \} . $$

The last inequality follows from
$$ \pi(\Delta^c)  = \hat{Q}(\Delta^c) +  \frac{\nu_1  \otimes \nu_2 (\Delta^c) }{1- \gamma}\leq$$
$$  \frac{\nu_1  \otimes \nu_2 (\Delta^c) }{1- \gamma}  \leq  \frac{\nu_1  \otimes \nu_2 (\Omega \times \Omega) }{1- \gamma}= \frac{\nu_1 (\Omega) \times \nu_2 (\Omega) }{1- \gamma} \,=\, \frac{(1-\gamma)^2 }{1- \gamma}=1-\gamma.$$ \qed
\medskip

Given a probability $\nu$ on the Bernoulli space $\Omega$ then the probability $\mu= \sigma^* (\nu)$ is by definition the one such that $\mu(A)= \nu (\,\sigma^{-1}(A)\,)$ for any Borel set $A$ on $\Omega$. We say that $\mu$ is invariant for the shift if $\mu= \sigma^* (\mu)$ .

\bigskip

\begin{proposition} Given $\mu_1$ and $\mu_2$ two probabilities over $\Omega$, then
$$ |\sigma^* (\mu_1) - \sigma^* (\mu_2)|_{tv}\, \leq |\mu_1 - \mu_2|_{tv}.$$
\end{proposition}

{\bf Proof:}

Note that by Kantorovich Duality (see \cite{Vi1})
$$ |\sigma^* (\mu_1) - \sigma^* (\mu_2)|_{tv}\,= \sup_{|f|\leq 1,\, f \,\,\text{measurable}\,} | \int f\, d \sigma^*\mu_1 -  \int f\, d \sigma^*\mu_2\,|=$$
$$ \sup_{|f|\leq 1,\, f \,\,\text{measurable}\,} | \int (f\circ \sigma )\, d \mu_1 -  \int (f \circ \sigma )\, d \mu_2\,|.$$

The functions of the form $(f\circ \sigma )$ with $|f|\leq 1$ is a smaller class than the set of functions of the form $g$ such that
$|g|\leq 1$.

From this follows the claim. \qed

In this way, the composition with the shift never increase the total variation norm of probabilities.

\medskip

\section{Estimates using the random time $T$} \label{pri}

Remember that points $x$ in $\Omega=\{1,2,...,d\}^\mathbb{N}$ are denoted by $x=(x_1,x_2,x_3,...)$.

\begin{definition} The coupling time $T$ is {\bf the} measurable function  $T:\Omega \times \Omega \to \mathbb{N}$ given by the random time
\begin{equation} \label{kli}T(x,y)= \inf \{n \, |\, x_m= y_m\, \text{for all}\,\, m \geq n\},
\end{equation}
for any $x,y\in \Omega$.
\end{definition}

This value can eventually be $\infty$.

$
T(x,y)=3$ when

$$x= (1,2,2,1,2,1,2,1,2,1,2...)$$
and
$$y=(2,2,1,1,2,1,2,1,2,1,2...).$$

An alternative way to define the  coupling time $T$ is given by
$$T(x,y)= \inf \{n \, |\, \sigma^n(x) = \sigma^n (y)\}, $$
for any $x,y\in \Omega$.

\medskip

Note that there is a difference in estimating the total variation of two probabilities in the state space $\{1,2,..,d\}$ and in the Bernoulli space $\{1,2,...,d\}^\mathbb{N}$. Now, we will consider results for each kind of setting.
\medskip

We will introduce two stochastic processes $X_n$ and $Y_n$, $n \in \mathbb{N},$ taking values in $S=\{1,2,...,d\}$.
We use the notation: for any $n$ consider the measurable function  $X_n$, such that 
$X_n: \{1,2,..,d\}^\mathbb{N} \to\{1,2,..,d\}$ such that $X_n(x)=x_n\in \{1,2,...,d\}$ if $x=(x_1,x_2,x_3,...,x_n,...)$. Consider now another process $Y_n: \{1,2,...,d\}^\mathbb{N} \to\{1,2,...,d\}$ defined in a similar way. For example, given a fixed line stochastic matrix $\mathcal{P}$, consider the Markov Process $X_n$, $n \in \mathbb{N}$, the associated Markov probability $P_1$ in $  \{1,2,...,d\}^\mathbb{N}$, with initial condition $\pi$, and the Markov Process  $Y_n$, $n \in \mathbb{N}$, the associated Markov probability $P_2$ in $  \{1,2,...,d\}^\mathbb{N}$, with initial condition $\nu$. 

\begin{example} \label{jjy} We are primarily  interested  in the case where the probability $\Psi$ in the next proposition is {\bf   the product probability $\Psi= P_1 \otimes P_2$ on $\{1,2,...,d\}^\mathbb{N}\times \{1,2,...,d\}^\mathbb{N}.$ This will be our working example in what follows in the next subsections. Our final goal is to estimate  the decay of $| P_1 ( X_n\in . )- P_2 ( Y_n\in  .) |_{tv}$, when $n \to \infty$. The $\Psi$ in next result is however more general.}
\end{example}

\medskip

\begin{proposition} \label{oo} Estimation in the state space - Suppose $X_n$ $Y_n$, $n \in \mathbb{N}$, are two Stochastic Processes, and $\mu_1^n$ and $\mu_2^n$, $n \in \mathbb{N}$ two probabilities on $\{1,2,...,d\}$. Suppose also that $\Psi$ is a probability on   $\{1,2,..,d\}^\mathbb{N}\times \{1,2,...,d\}^\mathbb{N}$, such that for a fixed  $n$,

for all $J\subset \{1,2...,d\}$, we have
\begin{equation} \label{hum}
\Psi (X_n\in J, Y_n\in \{1,2,...,d\})=\mu_1^n (J)
\end{equation} and
for all $J\subset \{1,2,...,d\}$, we have
\begin{equation} \label{hdo}\Psi ( X_n\in \{1,2,...,d\},Y_n\in J)= \mu_2^n (J) .\end{equation}

Then,
$$  \Psi ( X_n\in J ,Y_n\in S)- \Psi ( X_n\in S,Y_n\in J) \leq$$ 
$$  \Psi \{(x,y)\,|\,  T(x,y)> n \} .$$

Therefore,
\begin{equation} \label{esteaq} | \Psi ( X_n\in . )- \Psi ( Y_n\in  .) |_{tv} \leq 2 \, \Psi \{(x,y)\,|\,  T(x,y)> n \} .
\end{equation}

In other words,
\begin{equation} \label{esteaq91} | \mu_1^n- \mu_2^n |_{tv} \leq 2 \, \Psi \{(x,y)\,|\,  T(x,y)> n \} .
\end{equation}

\smallskip

{\bf Remark:} In the case $\Psi = P_1 \otimes P_2$, we get
for all $J\subset \{1,2...,d\}$, 
\begin{equation} \label{hum09}
\Psi (X_n\in J, Y_n\in \{1,2,...,d\})=P_1 (X_n \in J)= \mu_1^n (J),
\end{equation} and
for all $J\subset \{1,2,...,d\}$, 
\begin{equation} \label{hdo54}\Psi ( X_n\in \{1,2,...,d\},Y_n\in J)= P_2 (Y_n \in J)=\mu_2^n (J) .\end{equation}

\end{proposition}

{\bf Proof:} Given $J \subset \{1,2,...,d\}$

$$ \mu_1^n \,(J) - \mu_2^n \,(J) =$$
$$
\int \int_{x_n\in J}\, \Psi(dx,dy) -  \int_{y_n\in J} \int\, \Psi(dx,dy)=$$
$$[\, \int_{x_n=y_n} \int_{x_n \in J}\, \Psi(dx,dy) + \int_{x_n\neq y_n} \int_{x_n \in J}\, \Psi(dx,dy)\,] -$$
 $$[\, \int_{y_n\in J} \int_{x_n= y_n}\, \Psi(dx,dy)+
\int_{y_n \in J} \int_{x_n\neq y_n}\, \Psi(dx,dy)\,]=$$
$$  \int_{x_n\neq y_n} \int_{x_n \in J}\, \Psi(dx,dy)\,  - \int_{y_n \in J} \int_{x_n\neq y_n}\, \Psi(dx,dy)\leq $$
$$  \int_{x_n\neq y_n} \int_{x_n \in J}\, \Psi(dx,dy)\,\leq   \Psi \{(x,y)\,|\,  T(x,y)> n \} ,$$
because if $x$ and $y$ are such that $x_n\neq y_n$, then, $T(x,y) >n$. \qed

After some work, in the next sections, we will be able to use the estimate of the right-hand side of 
\eqref{esteaq},   showing exponential convergence to zero in $n$, of the sequence
$ | \Psi ( X_n\in . )- \Psi ( Y_n\in  .) |_{tv} $, when $\Psi$ is the product probability $\Gamma =(P_1 \otimes P_2)  $, as mentioned  in Example  \ref{jjy} (see also \eqref{hum09} and \eqref{hdo54}).

We point out that for many plans $\Psi$ we can have that $T=\infty$ almost everywhere. For some special ones, this is not true.

\medskip

\medskip

A more complex result is:

\begin{proposition} Estimation in the Bernoulli space - Suppose $\mu_1$ and $\mu_2$ are probabilities on $\{1,2,..,d\}^\mathbb{N}$. Given $\Psi \in \mathcal{C} (\mu_1,\mu_2) $, then for any $n$,
$$ |\,(\sigma^n)^*(\mu_1) - (\sigma^n)^*(\mu_2)\,|_{tv}\,\leq 2\, \Psi \{(x,y)\,|\,  T(x,y)> n \} .$$
\end{proposition}

{\bf Proof:}

Remember that if $(x,y)$ is such that for fixed $n$ we have $x_n\neq y_n$, then, $T(x,y) >n.$

Given a set $A\subset \Omega$ in $\mathcal{B},$

$$ \,(\sigma^n)^*(\mu_1)(A) - (\sigma^n)^*(\mu_2)(A) =$$
$$\mu_1\,\{x\,|(\sigma^n) (x) \in A\} - \mu_2\,\{y\,|(\sigma^n) (y) \in A\} \leq$$
$$
\int \int_{\{x \,|\, (\sigma^n) (x) \in A\}}\, \Psi(dx,dy) -  \int_{\{y \,|\,(\sigma^n) (y) \in A\}} \int\, \Psi(dx,dy)=$$
$$[\, \int_{\{(\sigma^n) (x)\neq (\sigma^n) (y)\} } \int_{\{x \,|\, (\sigma^n) (x) \in A\}}  \, \Psi(dx,dy) +$$
 $$\int_{\{ (\sigma^n) (x)= (\sigma^n) (y)\}}  \int_{\{x \,|\, (\sigma^n) (x) \in A\}}\, \Psi(dx,dy)\,] -$$
 $$[\, \int_{\{y \,|\, (\sigma^n) (y) \in A\}} \int_{\{(\sigma^n) (x)\neq (\sigma^n) (y)\}} \, \Psi(dx,dy)+$$
 $$
\int_{\{y \,|\, (\sigma^n) (y) \in A\}}  \int_{(\sigma^n) (x)= (\sigma^n) (y)} \, \Psi(dx,dy)\,]=$$
$$\, \int_{\{(\sigma^n) (x)\neq (\sigma^n) (y) \}} \int_{\{x \,|\, (\sigma^n) (x) \in A\}}  \, \Psi(dx,dy) - $$
$$  \int_{\{y \,|\, (\sigma^n) (y) \in A\}} \int_{\{(\sigma^n) (x)\neq (\sigma^n) (y)\}} \, \Psi(dx,dy)\leq $$
$$\, \int_{\{(\sigma^n) (x)\neq (\sigma^n) (y) \}} \int_{\{x \,|\, (\sigma^n) (x) \in A\}}  \, \Psi(dx,dy) \leq $$

$$ \int \int_{\{(x,y)\,|\, T(x,y) >n\}}  \Psi(dx,dy)\leq \Psi \{(x,y)\,|\,  T(x,y)> n \} .$$

Taking the supremum among all sets $A$ we get the claim. \qed

\bigskip

Suppose $X_n,\, n \in \mathbb{N}$ is a stochastic process over $S$ and $\Omega=S^\mathbb{N}.$
We assume that $X_n: \Omega \to S$ is such that $X_n(w)=w_n$, for any $n \in \mathbb{N}$, where $w=(w_1,w_2,...,w_n,..)\in \Omega$. On $\Omega$ we consider the sigma-algebra $\mathcal{A}$ generated by the cylinder sets (which is the same as the one generated by the open sets).
The stochastic process determines a probability $P$ on the sigma-algebra $\mathcal{A}$ of $\Omega$ (see \cite{Walk} or \cite{Lop}).

A {\bf stopping time} on $\Omega$ is a measurable function $T:\Omega \to \mathbb{N}$, such that, the set $A=\{w\,|\, T(w)=N\}$ depends
only  $X_1,X_2,,..,X_N$. In other words to know if $T(w)=N$ we just have to consider the string $ w_1,w_2,...,w_N.$

The coupling  time $T$ defined by \eqref{kli} is a random time but not a stopping time.
\medskip

We follow Hollander \cite{Hol}: a general program to estimate the  decay rate for the total variation distance of two processes.

Given a plan $\pi$ on $\Omega \times \Omega$, a  {\bf generic random  time} $\mathfrak{T}$ and a non-decreasing function $\psi: \mathbb{N} \to [0,\infty)$,  such that $\lim_{n \to \infty} \psi(n)= \infty$, assume that
$$ \int \psi\,(\, \mathfrak{T}(x,y)\,)\, \pi (dx,dy) <\infty.$$

Note that for any $n$
$$ \psi(n)\, \pi (\mathfrak{T}\,>\,n)\,\leq \int_{ \mathfrak{T}>n} \psi\,(\, \mathfrak{T}(x,y)\,)\, \pi (dx,dy).$$

The right hand side tends to zero when $n \to \infty$ by the dominated convergence theorem because $ \int \psi\,(\, \mathfrak{T}(x,y)\,)\, \pi (dx,dy) <\infty.$

Given $\epsilon$ suppose that $N$ is big enough such that for all $n>N$ we have $ \psi(n)\, \pi (\mathfrak{T}\,>\,n)\leq \epsilon.$

Suppose the plan $\pi$ is in  $\mathcal{C} (\mu_1,\mu_2) $, where  $\mu_1,\mu_2$ are probabilities on $S$.

Then, from the last proposition, we have that for $n$,
 \begin{equation} \label{estetam} |\,(\sigma^n)^*(\mu_1) - (\sigma^n)^*(\mu_2)\,|_{tv}\,\leq 2\, \pi \{(x,y)\,|\,  \mathfrak{T}(x,y)> n \}\leq 2\, \epsilon \frac{1}{ \psi(n)} .
 \end{equation}

Estimates of the form
\begin{equation} \label{esteaq1}| \pi ( X_n\in . )- \pi ( Y_n\in  .) |_{tv}  \leq 2 \, \pi \{(x,y)\,|\,  \mathfrak{T}(x,y)> n \}\leq 2\, \epsilon \frac{1}{ \psi(n)} 
\end{equation}
are also important and interesting.

In this way, one can get an estimation of the speed of convergence of the above difference by means of the random time $\mathfrak{T}$ and $\psi$. This depends on the plan $\pi$ we pick. The main point is the smart guess in choosing   the plan. All of the above also depends on $\psi(n)$ which can be of polynomial type $n^{-\gamma}$, $\gamma>0$, or exponential type $e^{-\lambda n}$, $\lambda>0$, depending of the problem.

The main problem in these kinds of questions is to estimate $\pi(\mathfrak{T}>n),$ where $n \in \mathbb{N}.$ This of course depends on $\pi$ and the random  time $\mathfrak{T}$ (could be a stopping time or not).

\bigskip
\section{Exponential convergence to equilibrium for finite state Markov chains}

\subsection{Estimates using  the stopping time  $T_1$} \label{eeuu}

We want to consider now  the following one: suppose $\mathcal{P}$ is a finite line stochastic matrix with all entries positive and $\lambda$ its unique invariant vector of probability.

Consider a  Markov chain $X_n$ with another initial condition $\nu$ and the same  transition matrix $\mathcal{P}$. We denote  the associated probability by $P$.

We want to show in Subsection \ref{sec} the existence of $0<\rho<1$, such that,
$$ | \lambda - P (X_n \in .) |_{tv} \leq \,2\, (1-\rho)^n.$$

Theorem \ref{pcc} will describe the  speed of convergence (which is exponential) to the equilibrium $\lambda$ when times $n$ go to infinity for the Markov chain $X_n$, $n \in \mathbb{N}.$ This result is the main goal of the next  subsections.

 \bigskip

\begin{definition} The  stopping time $T_1$ is {\bf the} measurable function  $T_1:\Omega \times \Omega \to \mathbb{N}$ given by
$$T_1(x,y)= \inf \{n \, |\, x_n= y_n\, \,\} $$
for any $x,y\in \Omega$.

\end{definition}

This value can eventually be $\infty$. Note that $T_1$ is in fact a stopping time.

Note that this coupling time is different from the other one denoted by $T$. Note also that $T_1(x,y)\leq T(x,y)$ for any $(x,y)$.

$T_1(x,y)=2$ when

$$x= (1,2,2,1,2,1,2,1,2,1,2...)$$
and
$$y=(2,2,1,1,2,1,2,1,2,1,2...).$$

An alternative way to define the  stopping time $T_1$ is given by
$$T_1(x,y)= \inf \{n \, |\, \sigma^n(x) = \sigma^n (y)\}, $$
for any $x,y\in \Omega$.

We point out that given a plan $\pi$ on $\Omega \times \Omega$ we have that
$$\pi ( T_1(x,y)>n  ) =   \pi \{ (x,y)\in \Omega \times \Omega\, |\, x_1\neq y_1,  x_2\neq y_2,..., x_n\neq y_n\}.$$

\bigskip
Now, we consider  two different stochastic processes $(X_n)_{n \in \mathbb{N}} $ and $(Y_n)_{n \in \mathbb{N}} $ taking values on $\{1,2,...,d\}$. In principle, in the general formulation,  we do not require that $X_n$ and $Y_n$ are Markov processes. We assume the initial time is $n=1$.
This is consistent with the notation $x=(x_1,x_2,x_3,...).$

 $\Gamma$ will denote the associated probability on  $\Omega \times \Omega$ describing the joint distribution of $X_n,Y_n$, $n \in \mathbb{N}$.
 That is,
 $$ \Gamma(  (A_1\times B_1)\times ...\times (A_n\times B_n)) =$$
 $$ \Gamma( (X_1,Y_1)\in (A_1,B_1),...,(X_n,Y_n)\in (A_n,B_n)).$$
 
For a fixed $n$ we will estimate $\Gamma \{ (x,y)\in \Omega \times \Omega \,|\, x_n\neq y_n\}$.

$\Gamma$ is {\bf not necessarily independent}.

We assume that
$$\Gamma \{ (x,y)\in \Omega \times \Omega\, | \,x_n= y_n\}\geq \rho>0.$$

In this way $\Gamma \{ (x,y)\in \Omega \times \Omega\, |\, x_n\neq y_n\}\leq (1-\rho)$.

\medskip

We  want to show that under {\bf some special conditions for $\Gamma$ (see \eqref{ee11})}, for any $n$ we have
\begin{equation} \label{ee}  \Gamma \{ (x,y)\in \Omega \times \Omega\, |\, x_1\neq y_1,  x_2\neq y_2,..., x_n\neq y_n\}\leq (1-\rho)^n.
\end{equation}

Consider a $d$ dimensional Stochastic Matrix $\mathcal{P}=(P_{i,j})_{i,j=1,2,...,d}$ with all entries positive and {\bf denote by $\rho$ the minimum value of $P_{i,j}$}. 

Consider two vector of initial probabilities $\lambda=(\lambda_1,...,\lambda_d) $ and $\nu=(\nu_1,...,\nu_d)$. Using the fixed matrix $\mathcal{P}$ they define respectively two different probabilities on $\Omega$ which are denoted by $P_1$ (using the initial vector of probability $\lambda$) and $P_2$ (using the initial vector of probability $\nu$). They generate respectively two different stochastic processes $(X_n)_{n \in \mathbb{N}} $ and $(Y_n)_{n \in \mathbb{N}} $ taking values on $\{1,2,...,d\}$. We assume the initial time is $n=1$.
This is consistent with the notation $x=(x_1,x_2,x_3,...).$

Consider the (plan) product probability  $ P_1 \otimes P_2$ on $\Omega \times \Omega$.
For a fixed $n$ we will estimate $( P_1 \otimes P_2) \{ (x,y)\in \Omega \times \Omega \,|\, x_n\neq y_n\}$ on Proposition \ref{aa}. {\bf More precisely, first we will show \eqref{ee}, when $\Gamma=P_1 \otimes P_2$, but a more general $\Gamma$, as in \eqref{ee1}, will be later considered}.

Note that for any others initial vector of probabilities $\tilde{\lambda}$ and $\tilde{\nu}$ we have
$$P_1 \otimes P_2 \{ (x,y)\in \Omega \times \Omega\, | \,x_n= y_n\}=\sum_{a_n=1}^d P_1 \otimes P_2\{ (x,y)\in \Omega \times \Omega\, |\, x_n= y_n=a_n\}=
$$
$$\sum_{a_n=1}^d\, [\,\sum_{j,a_1,..a_{n-1}=1}^d   \tilde{\lambda}_j   P_{j, a_1} \, P_{a_1,a_2} ...  P_{a_{n-1},a_{n}} \,]\, [\,\sum_{j,a_1,..a_{n-1}=1}^d   \tilde{\nu_j}   P_{j, a_1} \, P_{a_1,a_2} ..  P_{a_{n-1},a_{n}} \,]\geq$$
$$ \sum_{a_n=1}^d\, [\,\sum_{j,a_1,..a_{n-1}=1}^d   \tilde{\lambda}_j   P_{j, a_1} \, P_{a_1,a_2} ...  P_{a_{n-1},a_{n}} \,]\, [\,\sum_{j,a_1,..a_{n-1}=1}^d   \tilde{\nu}_j   P_{j, a_1} \, P_{a_1,a_2} ..  P_{a_{n-2},a_{n-1}}\,\rho \,]=$$
$$ \sum_{a_n=1}^d\, [\,\sum_{j,a_1,..a_{n-1}=1}^d   \tilde{\lambda}_j   P_{j, a_1} \, P_{a_1,a_2} ...  P_{a_{n-1},a_{n}} \,]\, \,\rho =\rho.$$

In this way $P_1 \otimes P_2 \{ (x,y)\in \Omega \times \Omega\, |\, x_n\neq y_n\}\leq (1-\rho)$.  A very important remark is that the above expression does not depend on the initial vector of probability  $\tilde{\lambda}$ and $\tilde{\nu}$. Indeed, just depends on the matrix $\mathcal{P}$.

For the {\bf independent Process} $(X_n,Y_n)$, $n \in \mathbb{N}$ we consider the {\bf  product probability }$ P_1 \otimes P_2$.
The Stochastic Process $(X_n,Y_n)_{n \in \mathbb{N}}$ taking values on $\{1,2,...,d\} \times \{1,2,...,d\}$,   such that
$$ (P_1 \otimes P_2) ( (X_1,Y_1)\in (A_1,B_1),...,(X_n,Y_n)\in (A_n,B_n)) =$$
$$P_1 (A_1\times ...\times A_n\times \{1,2,..,d\}^\mathbb{N})  \, P_2 (B_1\times  ...\times B_n\times \{1,2,..,d\}^\mathbb{N}) =$$
$$ (P_1 \otimes P_2) (\, (A_1\times ...\times A_n\times \{1,2,..,d\}^\mathbb{N})\, \times (B_1\times ...\times B_n\times \{1,2,..,d\}^\mathbb{N}) ,$$
is a Markov chain with a stochastic matrix
\begin{equation} \label{sono} \left(
\begin{array}{cc}
\mathcal{P} & 0\\
0 & \mathcal{P}
\end{array}
\right).
\end{equation}

The initial vector of probability is such that
\begin{equation} \label{ini} (P_1 \otimes P_2) ( X_1=i, Y_1=j)\,=\, \lambda_i\, \nu_j.
\end{equation}

$(P_1 \otimes P_2)$ is Markov probability on the space $ \{1,2,...,d\} \times \{1,2,...,d\}^\mathbb{N}.$
\medskip

Note that the event $X_1\neq Y_1$ is not independent of $X_2\neq Y_2$.

Assuming  the hypothesis  described above  we will show that:

\medskip

\begin{proposition} \label{aa}
\begin{equation} \label{ee10}
(P_1 \otimes P_2) (T_1>n)\,\,\leq\,\, (1-\rho)^n.
\end{equation}

{\bf Remark:} If $\Gamma$ is another probability in  $ \{1,2,...,d\} \times \{1,2,...,d\}^\mathbb{N}$  satisfying 
\begin{equation} \label{ee11}\Gamma  (T_1>n)= (P_1 \otimes P_2) (T_1>n),
\end{equation}
for all $n \in \mathbb{N}$, then
\begin{equation} \label{ee1}
\Gamma  (T_1>n)\,\,\leq\,\, (1-\rho)^n.
\end{equation}
\end{proposition}

{\bf Proof:}
Note that
$$   (P_1 \otimes P_2)\{ (x,y)\in \Omega \times \Omega\, |\, x_1\neq y_1,  x_2\neq y_2,..., x_n\neq y_n\}=$$
$$  (P_1 \otimes P_2) \{\, X_1\neq Y_1,  X_2\neq Y_2,..., X_n\neq Y_n\}=$$
$$ \frac{ (P_1 \otimes P_2) \{\, X_1\neq Y_1,  X_2\neq Y_2,..., X_n\neq Y_n\}}{\Gamma \{X_1\neq Y_1\}} \, (P_1 \otimes P_2)\{X_1\neq Y_1\} =$$
 $$  (P_1 \otimes P_2) \{\, X_1\neq Y_1,  X_2\neq Y_2,...,X_{n}\neq Y_{n}\,|\, X_1\neq Y_1\} \, (P_1 \otimes P_2)\{X_1\neq Y_1\}  \leq$$
$$  (P_1 \otimes P_2) \{ X_2\neq Y_2,...,X_{n}\neq Y_{n}\,|\, X_1\neq Y_1\} \, (1-\rho) . $$

\bigskip

We will need in this moment  the following property:
suppose $Z_n$, $n\in \mathbb{N}$,  is a Markov chain taking values in a finite set $E$ with transition matrix $\hat{\mathcal{P}}$. Consider also a certain initial condition $\tilde{\pi}$. This defines a Markov Probability $P$ of the space of paths $E^{\mathbb{N}}.$

Then,
$$ P(Z_2\in A_2,Z_3\in A_3,...Z_n\in A_n\,|\, Z_1\in A_1 )\,=\,$$
$$
\frac{P ( Z_1\in A_1, Z_2\in A_2,Z_3\in A_3,...,Z_n\in A_n  )}{P(Z_1\in A_1)}= $$
$$ \frac{   \sum_{j\in A_1} \,\, \tilde{\pi}_{a_j}\, [\,  \sum_{a_2\in A_2}\,...\,  \sum_{a_n\in A_n}\, \hat{\mathcal{P}}_{a_j,a_2} \,  \hat{\mathcal{P}}_{a_2,a_3}\,...\, \hat{\mathcal{P}}_{a_{n-1},a_n } ] }{ \sum_{j\in A_1} \tilde{\pi}_{a_j}}= $$
\begin{equation} \label{ti}
P_{\gamma}(Z_2\in A_2,Z_3\in A_3,...Z_n\in A_n),
\end{equation}
where $\gamma$ is  an initial vector of probability on $E$, such that, for $r\in A_1$ we have
\begin{equation}\label{ou}
\gamma_r= \frac{\tilde{\pi}_{a_r}}{\sum_{j\in A_1}\tilde{\pi}_{a_j}},
\end{equation}
and for $r$  which is not in $A_1$ we have  and $\gamma_r=0.$

The values $\hat{\mathcal{P}}_{i,j}\geq \rho$, $\forall i,j=1,2,...,d$.

The above property is a particular case of Prop 1.7, page 78 in \cite{RY} for a Markov Process $Z_n$ taking values  on $E$ which says
$$ P_q ( Z_{t+s} \in A\,|\, \mathcal{F}_t) = P_{Z_s} ( Z_t \in A),$$
where $  \mathcal{F}_t$ is the sigma algebra determined by the process from time 1 to $t$ and $q$ is an initial vector of probability on $E$.
This expression is sometimes called the Markov Property.

We will use (\ref{ti}) taking $Z_j=(X_j,Y_j)$, $A_j=\{X_j \neq Y_j\}$ and $E= \{1,2,...,d\} \times \{1,2,...,d\}$.

Therefore,
$$ (P_1 \otimes P_2) \{ (x,y)\in \Omega \times \Omega\, |\, x_1\neq y_1,  x_2\neq y_2,..., x_n\neq y_n\}=$$
$$  (P_1 \otimes P_2) \{ X_2\neq Y_2,...,X_{n}\neq Y_{n}\,|\, X_1\neq Y_1\} \, (1-\rho) . $$
$$  \tilde{\Gamma} \{X_2\neq Y_2,...,X_{n}\neq Y_{n}\} \, (1-\rho), $$
where $\tilde{\Gamma}$ is another plan (another  Markov probability on the space $ \{1,2,..,d\} \times \{1,2,..,d\}^\mathbb{N})$ due to the fact that we changed the initial condition as described in (\ref{ti}).

Now we use the fact that
$\tilde{\Gamma} \{ (x,y)\in \Omega \times \Omega\, |\, x_n\neq y_n\}\leq (1-\rho)$ and we get that
$$  (P_1 \otimes P_2)\{ (x,y)\in \Omega \times \Omega\, |\, x_1\neq y_1,  x_2\neq y_2,..., x_n\neq y_n\}\leq$$
$$  \tilde{\Gamma} \{X_2\neq Y_2,...,X_{n}\neq Y_{n}\} \, (1-\rho)\leq $$
$$  \tilde{\Gamma} \{ X_3\neq Y_3,...,X_{n}\neq Y_{n}\,|\, X_2\neq Y_2\}  (1-\rho)^2.$$

Proceeding by induction, we get \eqref{ee1}. 
\smallskip

The claim of the remark is obvious.

\qed

\bigskip

We want to combine \eqref{ee1} with  the estimate
\eqref{esteaq}  to show exponential convergence to zero in $n$ of the sequence  $| P_1 ( X_n\in . )- P_2 ( Y_n\in  .) |_{tv}  $ (see Subsection \ref{sec}).
\medskip

Consider as    in previous subsections, a $d$ by $d$ stochastic matrix $\mathcal{P}$ with all positive entries and we denote by $X_n$, $n \in \mathbb{N},$ and $Y_n$, $n \in \mathbb{N}$ two independent Markov Processes with values on $\{1,2,...,d\}$ respectively associated to the same matrix $\mathcal{P}$; and $\rho$ as before. We denote $S= \{1,2,...,d\}$.

Consider an initial probability $\lambda$ for $X_n^\lambda$ and $\nu$ for $Y_n^\nu$. This will define probabilities on
$\{1,2,...,d\}^\mathbb{N}$ which we will denote respectively by $P_1$ and $P_2$. We will denote $P_1^j$ the probability we get from the Markov Process defined by the transition matrix $\mathcal{P}$ and the  initial vector of probability $\delta_{j}$, where $j\in\{1,2,...,d\}.$

We consider first a Stochastic Process $V_n$, $n \in \mathbb{N},$ with values on  $S^2=\{1,2...,d\}^2$
of the form $V_n= (X_n^\lambda,Y_n^\nu).$ Our  assumption is that  the correspondent {\bf probability $\tilde{P}$} on $\Omega \times \Omega$  is the product probability $\tilde{P}=P_1 \otimes P_2$. That is, the processes $X_n^\lambda$ and $Y_n^\nu$ are independent. 

We consider another  Stochastic Process $U_n$, $n \in \mathbb{N},$ with values on  $\{1,2...,d\}^2$
of the form $U_n= (X_n,Y_n).$ This will define    {\bf   another probability $\hat{P}$} on $(\{1,2...,d\}^2)^\mathbb{N}$.
In this case, the processes $X_n$ and $Y_n$ are  not independent.

The main issue in this subsection is to define a probability $\hat{P}$ on $(S \times S)^\mathbb{N}$ such that for the probability $\tilde{P}= P_1 \otimes P_2$, it is valid the expression

$$ \, \hat{P} \{(x,y)\,|\,  T> n \} =  \, \tilde{P} \{(x,y)\,|\,  T_1> n \} .$$

Then, as  $\tilde{P}(T_1>n)\,\,\leq\,\, (1-\rho)^n$ from Proposition \ref{aa}, it will follow a similar property  $\tilde{P} \{(x,y)\,|\,  T> n \}\leq\,\, (1-\rho)^n$. The main result will be obtained in Subsection \ref{sec}: using \eqref{esteaq} to show (see \eqref{uff}) that
\begin{equation} \label{uff23}
| P_1( X_n\in . )- P_2 ( Y_n\in  .) |_{tv}  \leq \,2\, (1-\rho)^n,
\end{equation}
which implies, for instance, exponential convergence to equilibrium for a Markov Chain $X_r$, $r \in \mathbb{N}$, taking values in $S=\{1,2,...,d\}$, when all entries of $\mathcal{P}$ are positive.

For a fixed $n$ we have to define $\hat{P}$ is sets of the form
\begin{equation} \label{ou49} \,(\,A_1 \times B_1\,) \times (\,A_2 \times B_2\,)\times...\times (\,A_n \times B_n\,)\,\times (S\times S)^\mathbb{N}, 
\end{equation}
where $A_t,B_t\subset S$, $t=1,2...,n$.

Now, we will indicate how the probability  $\hat{P}$ is defined: suppose $T_1:  \{1,2...,d\}^\mathbb{N}\times \{1,2...,d\}^\mathbb{N}\to \mathbb{N}$ is given by
$$T_1(x,y)= \inf \{n \, |\, x_n= y_n\, \,\} $$
for any $x,y\in \{1,2...,d\}^\mathbb{N} $.

The family of sets $  \mathfrak{B}_k =\{(x,y)| T_1(x,y)=k\}$, $k \in \mathbb{N}$, is a partition of $(S \times S)^\mathbb{N}.$

For fixed $k \in \mathbb{N}$ and $j=1,2,...,d$, consider the set 
$$\mathfrak{B}_k^j=\{ (x,y)\in \mathfrak{B}_k \, \text{and}\,x_k=y_k=j  \}.$$

The family $\mathfrak{B}_k^j$, $j=1,2,...,d$, is a partition of $\mathfrak{B}_k.$

For fixed $k \in \mathbb{N}$ and $j\in S$, consider the set
\begin{equation} \label{entama}
\mathfrak{R}_{k,j} =\{(x,y)\in \mathfrak{B}_k^j, \,\text{and}\,\, (X_m,Y_m)\cap \Delta = \emptyset \,,\,\,
 \text{for some}\,\, k<m \}.
\end{equation}

We denote
\begin{equation} \label{gov} \mathfrak{R} = \cup_{k=1}^\infty \cup_{j=1} ^d\, \mathfrak{R}_{k,j},\, \text{
and we declare}\, \hat{P} (  \mathfrak{R})=0.
\end{equation}

Given $k,j$, it remains to define the probability $\hat{P}$ for subsets $V\subset \mathfrak{B}_k^j \cap \mathfrak{R}_{k,j}^c$.

For fixed $k,j$, we will define $\hat{P}$ for sets $V$ of the form
\begin{equation} \label{entama1}((\,A_1 \times B_1\,) \times (\,A_2 \times B_2\,)\times...\times (\,A_n \times B_n)\,\times (S\times S) ^\mathbb{N}) \cap  \mathfrak{B}_k^j\cap \mathfrak{R}_{k,j}^c.
\end{equation}
They have to be defined in a way to compensate the assumption \eqref{gov}.

For simplification one can assume all the sets  on \eqref{entama1} are of the form 
\begin{equation} \label{lor}A_t=\{a_t\}, B_t=\{b_t\}, t=1,2,...,n.
\end{equation}

\bigskip

I) For subsets   $V\subset \mathfrak{B}_k^j \cap \mathfrak{R}_{k,j}^c$, and $n>k=T_1$,  we declare  that
$$\hat{P} [\,(\,A_1 \times B_1\,) \times (\,A_2 \times B_2\,)\times...\times (\,A_n \times B_n\,)\times  (S\times S)^\mathbb{N}\,]=        $$
\begin{equation}  \label{ou}
\tilde{P} [\,(\,A_1 \times B_1\,) \times...  \times(\,A_{k-1} \times B_{k-1} \,)\times (\{j\} \times S) \times  (\,A_{k+1} \times S\,)\times...\times  (\,A_n \times S\,)\times  (S\times S)^\mathbb{N}\,]       
.
\end{equation}

\begin{example}
A particular example, when  $d=4$, $T_1=k=3$, $j=3$ and $n=5$, and $V\subset \mathfrak{B}_3^3 \cap \mathfrak{R}_{3,3}^c$ is of the form \eqref{lor}
$$\hat{P} [\,(\,(\,1 ,2\,) \times (\,2 , 1 \,)\times (\,3,3\,)\times \,(2,2\,)\times (4,4)\, \times  (S\times S)^\mathbb{N})\,]=        $$
$$\hat{P} [\,(\,1 ,2\,) \times (\,2 , 1 \,)\times (\,3 \times S\,)\times \,(2,S\,)\times (4,S)\, \times  (S\times S)^\mathbb{N})\,]=        $$
$$  P_1 (1  \times 2 \times 3 \times \,2 \times  4 \times S^\mathbb{N}) \, P_2 ( 2 \times 1 \times S \times \,S \times S \times S^\mathbb{N})= $$
\begin{equation}  \label{Val}\tilde{P} [\,(\,1 ,2\,) \times (\,2 , 1\,)\times (3 \times S) \times \, \,(2,S\,)\times (4,S)\times (S\times S)^\mathbb{N}\,].        \end{equation}

Note that from \eqref{gov}
$$\hat{P} [\,(\,(\,1 ,2\,) \times (\,2 , 1 \,)\times (\,3,3\,)\times \,(2,3\,)\times (4,4)\, \times  (S\times S)^\mathbb{N})\,]=  0.      $$

In consonance with \eqref{gov},  we get  for $V\subset \mathfrak{B}_3^3 \cap \mathfrak{R}_{3,3}^c$ of the form, 
$$\hat{P}(S \otimes S \otimes (3 \times S) \otimes( 2\times S)  \otimes (4 \times S)\otimes  (S\times S)^\mathbb{N})=$$
$$ \hat{P}( S \otimes S \otimes(3 \times 3) \otimes( 2\times 2) \otimes (4 \times 4) \otimes  (S\times S)^\mathbb{N}).$$

The above express the property that once the two paths meet together for the first time, they remain together in the future times. It is also consistent with 
$$ \hat{P}( S \otimes S \otimes(3 \times 3) \otimes( 2\times \{1,3,4\}) \otimes (4 \times \{1,2,3\}) \otimes  (S\times S)^\mathbb{N})=0.$$

It follows from \eqref{ou}  that for $V$ in  $\mathfrak{B}_5= \cup_{j=1}^d \mathfrak{B}_5^j$, summing in   all $j=1,2,...,d$, in \eqref{Val}
$$\hat{P} [\,(\,(\,1 ,2\,) \times (\,2 , 1 \,)\times (\,S \times S\,)\times \,(2,2\,)\times (4,4)\, \times  (S\times S)^\mathbb{N})\,]=        $$
$$\hat{P} [\,(\,(\,1 ,2\,) \times (\,2 , 1 \,)\times (\,S \times S\,)\times \,(2 \times S\,)\times (4 \times S)\, \times  (S\times S)^\mathbb{N})\,]=        $$
\begin{equation}  \label{ou97}\tilde{P} [\,(\,1 ,2\,) \times (\,2 , 3 \,)\times (S \times S) \times \,\,(2 \times S\,)\times (4 \times S)\times (S \times S)^\mathbb{N}\,].        \end{equation}

Moreover from the above reasoning, for $V$ in  $\mathfrak{B}_5= \cup_{j=1}^d \mathfrak{B}_5^j$ of the below form, summing in   all $j=1,2,...,d$, 
$$\hat{P} [\, (\,S \times S\,) \times (\,S \times S\,)\times (\,S \times S\,)\times \,(2,2\,)\times (4,4)\, \times  (S\times S)^\mathbb{N})\,]=        $$
$$\hat{P} [\, (\,S \times S\,) \times (\,S \times S\,)\times (\,S \times S\,)\times \,(2\times S\,)\times (4 \times S)\, \times  (S\times S)^\mathbb{N})\,]=        $$
\begin{equation}  \label{ou971}\tilde{P} [\,(\,\,S \times S\,) \times (\,S \times S\,)\times (S \times S) \times \,\,(2 \times S\,)\times (4 \times S)\times (S \times S)^\mathbb{N}\,].        \end{equation}

 $\diamondsuit$ 
\end{example}

In all of the above cases  we get that for $V\subset \mathfrak{B}_k^j \cap \mathfrak{R}_{k,j}^c$
\begin{equation} \label{ou249}\hat{P}(V)=\tilde{P}(V).\end{equation}

\medskip

II) When $k=n$, and $A_n=\{j\}=B_n$, and $T_1=n$, we set
$$\hat{P} [\,(\,A_1 \times B_1\,) \times (\,A_2 \times B_2\,)\times...\times (\,A_n \times B_n\,)\times  (S\times S)^\mathbb{N}\,]=        $$
$$P_1 ( A_1 \times... \times A_{k-1}\times \,{j}\times S^\mathbb{N}) \, \, P_2 ( B_1 \times...  \times B_{k-1} \times \,S\times S^\mathbb{N})\, $$
\begin{equation}  \label{ou32} \tilde{P} ((\,A_1 \times B_1\,) \times (\,A_2 \times B_2\,)\times...\times  (A_{n-1}\times B_{n-1}) \times (\,{j} \times S \,)\times  (S\times S)^\mathbb{N}  ).
\end{equation}

\begin{example}

A particular  example, when $d=4$,  $T_1=k=4$, $j=2$ and $n=4$, and $V\subset \mathfrak{B}_4^2 \cap \mathfrak{R}_{4,2}^c$ is of the form \eqref{lor}
$$\hat{P} [\,(\,1 ,2\,) \times (\,2 , 1 \,)\times (\,3,2\,)\times \,(2,2\,)\times  (S\times S)^\mathbb{N}\,]=        $$
$$ \hat{P} [\,(\,1 ,2\,) \times (\,2 , 1 \,)\times (\,3,2\,)\times \,(\{2\} \times S\,)\times  (S\times S)^\mathbb{N}\,]  $$
$$  P_1 (1  \times 2 \times 3 \times \,2 \times S^\mathbb{N}) \, P_2 ( 2 \times 1 \times 2 \times \,S  \times S^\mathbb{N})= $$
$$(P_1 \otimes P_2)[\,(\,1 ,2\,) \times (\,2 , 1 \,)\times (\,3,2\,)\times \,(2,S\,)\times  (S\times S)^\mathbb{N}\,]=        $$
$$ \tilde{P}[\,(\,1 ,2\,) \times (\,2 , 1 \,)\times (\,3,2\,)\times \,(2,S\,)\times  (S\times S)^\mathbb{N}\,].        $$

Note that
$$\hat{P} [\,(\,(\,1 ,2\,) \times (\,2 , 1 \,)\times (\,3,2\,)\times \,(2,3)\,\times  (S\times S)^\mathbb{N}))\, \,]= 0       $$

$\diamondsuit$ 

\end{example}

\smallskip

In all of the above cases,  when $V\subset \mathfrak{B}_k^j \cap \mathfrak{R}_{k,j}^c$ we get that 
\begin{equation} \label{ou249}\hat{P}(V)=\tilde{P}(V).\end{equation}

\smallskip

III) Now, for subsets   $V \subset \mathfrak{B}_k^j$, where $n<k= T_1(x,y)$,  we declare  that
$$\hat{P} [\,(\,A_1 \times B_1\,) \times (\,A_2 \times B_2\,)\times...\times (\,A_n \times B_n\,)\times  (S\times S)^\mathbb{N}\,]=        $$
\begin{equation}  \label{ou331}\tilde{P} [\,(\,A_1 \times B_1\,) \times ...\times (\,A_n \times B_n\,)\times  (S\times S)^{ k-n-1} \times (\{j\} \times S) \times (S\times S)^\mathbb{N}\,]
\, \, .
\end{equation}

\begin{example}
For example, when $d=4$,   $T_1=k=4$, $j=3$ and $n=2$, and $V\subset \mathfrak{B}_4^3 \cap \mathfrak{R}_{4,3}^c$ is of the form \eqref{lor}
$$\hat{P} [\,(\,1 ,2\,) \times (\,2 , 3 \,)\cap (\mathfrak{B}_4^3 \cap \mathfrak{R}_{4,3}^c)=$$
$$\hat{P} [\,(\,1 ,2\,) \times (\,2 , 3 \,)\times (S \times S) \times \,(3,3\,)\times (S \times S)^\mathbb{N}\,]=        $$
$$\hat{P} [\,(\,1 ,2\,) \times (\,2 , 3 \,)\times (S \times S) \times \,(\{3\} \times S\,)\times (S \times S)^\mathbb{N}\,]=        $$
$$ \tilde{P}  [\,(\,1 ,2\,) \times (\,2 , 3 \,)\times  (S \times S)\times \,(\{3\} \times S \,)\times (S \times S)^\mathbb{N}\,]=$$
$$  P_1 (1  \times 2 \times S \times \,3 \times S^\mathbb{N}) \, P_2 ( 2 \times 3 \times \,S \times S \times S^\mathbb{N})=$$
\begin{equation}  \label{ou87} (P_1 \otimes
P_2) [\,(\,1 ,2\,) \times (\,2 , 3 \,)\times  (S \times S)\times \,(\{3\} \times S \,)\times (S \times S)^\mathbb{N}\,].
\end{equation}

It follows that for $V$ in  $\mathfrak{B}_4= \cup_{j=1}^d \mathfrak{B}_4^j$, summing in   all $j=1,2,...,d$ in \eqref{ou87} 
$$\hat{P} [\,(\,1 ,2\,) \times (\,2 , 3 \,)\times (S \times S) \times \,(S \times S\,)\times (S \times S)^\mathbb{N}\,]=        $$
\begin{equation}  \label{ou97}\tilde{P} [\,(\,1 ,2\,) \times (\,2 , 3 \,)\times (S \times S) \times \,(S \times S\,)\times (S \times S)^\mathbb{N}\,].        \end{equation}

Note that,
$$\hat{P} [\,(\, (\,1 ,2\,) \times (\,2 , 3 \,)\times (S \times S)\times \,(3,2\,)\times (S \times S)^\mathbb{N}\,)\cap\mathfrak{B}_3^4\,]=0        $$

$\diamondsuit$ 

\end{example} 

In all of the above cases, when $V\subset \mathfrak{B}_k^j \cap \mathfrak{R}_{k,j}^c$
\begin{equation} \label{ou249}\hat{P}(V)=\tilde{P}(V).\end{equation}

\medskip

Considering the probability $\hat{P}$ for subsets $V$ (of the form \eqref{ou49}) inside $ \mathfrak{B}_k^j \cap \mathfrak{R}_{k,j}^c$, for different $k, n$, and $j=1,2,...,d$, we were able to define the probability $\hat{P}$ in all  $ (S\times S)^\mathbb{N}.$ For the definition of the probability $\hat{P}$ we use in a fundamental way the stopping time $T_1$.

One can say in an elusive way that $\hat{P}$ describes the following process:  two particles located on $S$ evolve in time in an independent way (a Markov Chain $(X_t,Y_t) $, $t \in \mathbb{N},$ taking values in $S\times S$), and when they meet, they stay together in the future according to the law of the Markov chain in $S$.

\medskip

\begin{proposition} \label{coli} For any $k\in \mathbb{N}$ we get $\hat{P} (T_1 = k)= \tilde{P} (T_1 = k)$.

\end{proposition}

\begin{proof} For a given fixed $k$ we have to consider different sets (depending on $n$)
\begin{equation} \label{ou489} (\,(\,A_1 \times B_1\,) \times (\,A_2 \times B_2\,)\times...\times (\,A_n \times B_n\,)\,\times (S\times S)^\mathbb{N} \,)\cap \mathfrak{B}_k, 
\end{equation}
where $A_t,B_t\subset S$, $t=1,2...,n$.

In cases I)  and II),
it follows that in $\mathfrak{B}_k= \cup_{j=1}^d \mathfrak{B}_k^j$, summing in   all $j=1,2,...,d$, in a similar way as in \eqref{ou} and  \eqref{ou32}, when $T_1=k \leq n$,
 we get
$$\hat{P} [\,(\,A_1 \times B_1\,) \times (\,A_2 \times B_2\,)\times...\times (\,A_n \times B_n\,)\times  (S\times S)^\mathbb{N})\,\cap \mathfrak{B}_k\,]=  $$
$$ \hat {P} ( \{T_1=k\} \cap (\,(A_1 \times B_1\,) \times ...\times (\,A_n \times B_n\,)  \times (S \times S)^{n-k} \times (S \times S)^\mathbb{N}\,)=$$
\begin{equation} \label{ou48} \tilde {P} ( \{T_1=k\} \cap (A_1 \times B_1\,) \times ...\times (\,A_n \times B_n\,)  \times (S \times S)^{n-k} \times (S \times S)^\mathbb{N}\,).   
\end{equation}

\eqref{ou971} is a particular example of the above.

In case III),
it follows that in $\mathfrak{B}_k= \cup_{j=1}^d \mathfrak{B}_k^j$, summing in   all $j=1,2,...,d$, in a similar way as in  \eqref{ou331}, when $T_1=k > n$,
 we get
 $$\hat{P} [\,(\,A_1 \times B_1\,) \times (\,A_2 \times B_2\,)\times...\times (\,A_n \times B_n\,)\times  (S\times S)^\mathbb{N}\cap \mathfrak{B}_k\,]=  $$
\begin{equation} \label{ou483} \tilde {P} ( \{T_1=k\} \cap (A_1 \times B_1\,) \times ...\times (\,A_n \times B_n\,)  \times (S \times S)^{n-k} \times (S \times S)^\mathbb{N}\,).   
\end{equation}

\eqref{ou97} is a particular example of the above.
\smallskip

The claim follows from \eqref{ou483} and \eqref{ou48}.

\end{proof}

\medskip

We want to show that the hypothesis of Proposition \ref{oo} are true for $\hat{P}$ (and therefore, we will be able to get the consequences obtained in the last subsections, in particular the claim of Proposition \ref{aa}, regarding \eqref{ee1}). The interplay between $T_1$ and $T$ will play an important role.

\medskip

Remember that we denoted $\tilde{P}= P_1 \otimes P_2$.

\medskip

One can show the next proposition, but it is not relevant for our reasoning.

\begin{proposition} \label{teq1}
For fixed $j$and $n$, 
\begin{equation} \label{tama}\hat {P} (X_n=j,  Y_n\in S) = \tilde {P} (X_n=j,  Y_n\in S)= P_1 (X_n=j ) .
\end{equation}
\end{proposition}
\bigskip

In a similar way, one can show that:

\begin{proposition} \label{teq3}
For fixed $j$ and $n$, 
\begin{equation} \label{tama7}\hat {P} (X_n\in S,  Y_n=j) \geq \rho .
\end{equation}
\end{proposition}

\medskip

{\bf Remark:} Note that it is not true that 
for all $j$ and $n$, 
\begin{equation} \label{tama765}\hat {P} (X_n \in S,  Y_n=j) = P_2 (Y_n=j  ) .
\end{equation}

\medskip

In order to take advantage of \eqref{ee1}  (see hypothesis \eqref{ee11}), when $\Gamma=\hat{P}$, 
we will need  the following result:

\bigskip

\begin{proposition}

The probability $\hat{P}$ on $(\{1,2...,d\}^2)^\mathbb{N}$ which is a plan on the product space
$\{1,2,...,d\}^\mathbb{N}\times \{1,2,...,d\}^\mathbb{N}$, for fixed $m$ satisfies
\begin{equation} \label{tamap}
\hat{P} (T > m) = \hat{P} (T_1 > m)=  \tilde{P} (T_1 > m),
\end{equation}

where $\tilde{P}=P_1 \otimes P_2$ is the product probability.

\end{proposition}

{\bf Proof:} We claim that $\hat{P} (T > m) = \hat{P} (T_1 > m)$

We consider the partition of $\{1,2,...,d\}^\mathbb{N}\times \{1,2,...,d\}^\mathbb{N}$ in the sets $\mathfrak{B}_r$ and $\mathfrak{B}_r^c$.

It follows from \eqref{entama} and \eqref{gov} that (up to $\hat{P}$   measure zero),  for
 $(x,y) \in (S \times S)^\mathbb{N}$
$$ \, T_1(x,y) =r  \Longleftrightarrow T (x,y) =r .$$

Then, 
$$\hat{P}( \{T > m\}) =\hat{P}( \{T > m\} \cap  \mathfrak{R}^c)=  \hat{P} (\{T_1 > m\} \cap  \mathfrak{R}^c)= \hat{P} (\{T_1 > m\} .$$

Finally,  from Proposition \ref{coli}, for any $k$ we get $\hat{P} (T_1 = k)= \tilde{P} (T_1 = k)$. Then, 
 $$\hat{P} (T_1 > m)= \sum_{k >m} \hat{P} (T_1 = k) =\sum_{k >m} \tilde{P} (T_1 = k) =\tilde{P} (T_1>n  ).$$

 \qed

\subsection{Speed estimation on total variation} \label{sec}

The probabilities $P_1$ and $P_2$ on $\{1,2,...,d\}^\mathbb{N}$ and $X_n, Y_n$, $n \in \mathbb{N}$, were described above. We want to take advantage of (\ref{esteaq}),
(\ref{ee1}) and (\ref{tama}).

Given $\Psi$, 
we get before estimates of the form (\ref{esteaq}) for $T$ satisfying expression (\ref{kli}), that is, for $X_n$ and $Y_n$ independent
$$ | \Psi ( X_n\in . )- \Psi ( Y_n\in  .) |_{tv} \leq 2 \, \Psi \{(x,y)\,|\,  T(x,y)> n \},$$
when considering the random time
$$T(x,y)= \inf \{n \, |\, x_m= y_m\, \text{for all}\, m \geq n\}, $$

and where $\Psi$ satisfies the hypothesis \eqref{hum09} and \eqref{hdo54} of Proposition \ref{oo}.
That is, we are taking  $\Psi=\tilde{P}= P_1 \otimes P_2$.

For the validity of equation \eqref{ee1} we take $\Gamma= \hat{P}.$

Then, we get from (\ref{tamap}) (see hypothesis \eqref{ee11}) that
$$  | \tilde{P} ( X_n\in . )- \tilde{P} ( Y_n\in  .) |_{tv}\leq 2 \, \tilde{P} \{(x,y)\,|\,  T> n \}=  2 \, \hat{P} \{(x,y)\,|\,  T_1> n \} .$$

Now, from (\ref{ee1}) (and hypothesis \eqref{ee11})  we get that
$\hat{P} (T_1>n)\,\,\leq\,\, (1-\rho)^n.$

As  get that  $\tilde{P} ( X_n\in . )=  P_1 ( X_n\in . )$ and $ \tilde{P} ( Y_n\in . )=  P_2 ( Y_n\in . )$. Then,  it  follows that
$$  | P_1 ( X_n\in . )- P_2 ( Y_n\in  .) |_{tv} =$$
\begin{equation} \label{uff}
| \tilde{P} ( X_n\in . )- \tilde{P} ( Y_n\in  .) |_{tv} \leq \,2\, (1-\rho)^n.
\end{equation}

Note that  $P_1 (X_n \in \,.)$ and $P_2 (Y_n \in \,.)$ are probabilities on the state space $S$.

The bottom line is: suppose $\lambda=(\lambda_1,\lambda_2,...,\lambda_d)$ is the stationary vector for the $d$ by $d$ matrix $\mathcal{P}$ (which has all entries positive) and $\nu$ is another initial vector of probability on $\{1,2,...,d\}$. This defines respectively two probabilities on  $\{1,2,...,d\}^\mathbb{N}$ which we denote
$P_1$ and $P_2$. The vector of probability  $\lambda$ is unique. We also denote, respectively,  $X_n$, $n \in \mathbb{N},$  for the first Markov Process and $Y_n$, $n \in \mathbb{N},$ for the second. We assume for the definition of $P_1$ we use the initial condition $\lambda$, and  for $P_2$ we use another initial vector of probability $\nu$.

Note that for any $n$ and $J\subset S$ we have $ P_1 (X_n \in J)=P_1 (X_1 \in J) $ because $\lambda$ is stationary. In other words $P_1 (X_n =j\,)= \lambda_j$, $j=1,2,...,d$, $\forall n \in \mathbb{N}.$

In this way, we are analyzing the time evolution of two probabilities on the space state $S$, and from (\ref{uff}), we get:
\medskip

\begin{theorem} \label{pcc}
Suppose the transition matrix $\mathcal{P}$
has all entries positive and $\lambda \, \mathcal{P}=\lambda$, where $\lambda$ is the initial stationary vector of probability on $S$. Assume $Y_n$, $n \in \mathbb{N}$, is another Markov probability associated with the stochastic matrix $\mathcal{P}$, but with a different initial vector of probability $\nu$.
Then, for any $n$ we have
$$ | \lambda - P_2 (Y_n \in .) |_{tv} \leq \,2\, (1-\rho)^n,$$
and this describes the speed of convergence to the equilibrium $\lambda$ when times goes to infinity for a chain $Y_n$ with initial condition $\nu$ and matrix transition $\mathcal{P}$.
\end{theorem}

\bigskip

\section{The $\bar{d}$ distance} \label{dbar}

Given two probabilities on $\mu,\nu$ in the set $\Omega= \{1,2\}^\mathbb{N},$ consider the set $\mathcal{C}(\mu,\nu)$ of plans $\lambda$ in $\Omega \times \Omega$ such the projection in the first coordinate is $\mu$ and in the second is $\nu$.

\begin{definition} - A joining is a probability on $\mathcal{C}(\mu,\nu)$ which is invariant for the dynamical system $T:  \Omega \times \Omega\to \Omega \times \Omega$ given by $T(x,y)= (\sigma(x), \sigma(y))$. The set of such joinings is denoted by $J(\mu,\nu)$.
\end{definition}

A general reference for joinings is \cite{Glas}.

\bigskip

\begin{definition}
Denote $\{1,2\}^n=Q_n$ and consider the $d_n$-Hamming metric in $Q_n$ defined by
$$d_n (x,y) =\frac{1}{n} \, \sum_{j=1}^n \, I_{\{x_j\neq y_j\}}.$$

\end{definition}

Points $x$ in $\Omega $ are denoted by $x=(x_0,x_1,x_2,...).$
\medskip

\begin{definition}
For two $\sigma$-invariant probabilities $\mu$ and $\nu$ we define the distance
$$\bar{d}(\mu,\nu)= \inf_{\lambda \in J(\mu,\nu)} \int \, I_{\{x_0\neq y_0\}}\lambda (dx, dy)=$$
$$  \inf_{\lambda \in J(\mu,\nu)}\, \lim_{n \to \infty} \int d_n (x,y) \lambda (dx , dy)=$$
$$ \lim_{n \to \infty}  \inf_{\lambda \in J(\mu,\nu)}\,  \int d_n (x,y) \lambda (dx , dy).$$

\end{definition}

In the above equalities, we used the ergodic theorem (see \cite{Walk}).

The above value tries to minimize the asymptotic mean disagreement between the symbols of the paths for a given plan.

\medskip

We point out that the $\bar{d}$ has a dynamical content. Several dynamical properties are preserved under limit over the $\bar{d}$ distance (see \cite{CQ}). In this  paper, properties related to Bernoullicity,  $g$-measures, and Gibbs states are considered. For instance, the map that takes a potential to its equilibrium state is continuous with respect to the $\bar{d}$ distance (see Theorem 3 in \cite{CQ}).

\medskip

\begin{theorem} \label{CC}
Suppose $\mu$ is the independent Bernoulli probability associated to $p_1,p_2$ and
$\nu$ is the independent Bernoulli probability associated to $q_1,q_2$.

Suppose $p_1\leq q_1$.
 Then,
 $$\bar{d}(\mu,\nu)=\,q_1\,-\,p_1= \frac{1}{2} (|q_1-p_1| + |q_2-p_2|\,).$$
\end{theorem}

 {\bf Proof:}

 Given $\lambda \in J(\mu,\nu)$ we have
$$   \int \, I_{\{x_0\neq y_0\}}\lambda (dx, dy)=$$
$$   \int_{[1]\times[1]} \, I_{\{x_0\neq y_0\}}\lambda (dx, dy)+   \int_{[1]\times[2]} \, I_{\{x_0\neq y_0\}}\lambda (dx, dy)+$$
$$   \int_{[2]\times[1]} \, I_{\{x_0\neq y_0\}}\lambda (dx, dy)+   \int_{[2]\times[2]} \, I_{\{x_0\neq y_0\}}\lambda (dx, dy)=$$
$$   \int_{[1]\times[2]} \, I_{\{x_0\neq y_0\}}\lambda (dx, dy)+   \int_{[2]\times[1]} \, I_{\{x_0\neq y_0\}}\lambda (dx, dy).$$

 We denote $a_{i,j}= \int_{[i]\times[j]} \, I_{\{x_0\neq y_0\}}\lambda (dx, dy)$.

 %Note that for $i\neq j$ we have $a_{i,j}=\lambda([i]\times [j]).$

 We have the relations
 $$ \lambda_{1,1}+ \lambda_{2,1}= \nu[1]=q_1 \,,\,\,\,  \lambda_{1,2}+ \lambda_{2,2}=\nu[2]=q_2\, ,$$
 $$\,\,\, \lambda_{1,1}+ \lambda_{1,2}=\mu[1]=p_1  \,,\,\,\,  \lambda_{2,1}+ \lambda_{2,2}=\mu[2]=p_2\, .
 $$

 Denote by $t=\lambda_{1,1}$, then

 $$\lambda_{1,2}= p_1 - t,$$

 $$\lambda_{2,1}= q_1 - t,$$

 and

$$\lambda_{2,2}= p_2  - q_1 + t ,$$

 We  have to minimize $a_{12} + a_{21}\,=\,p_1 + q_1 - 2 t$ given the above linear constrains. This is a linear maximization problem, and the largest possible value
 of $t$ will be the solution.

 From the fact that the $\lambda_{i,j}$ are non-negative, we get that $t\leq q_1$ and $t\leq p_1$. But as we assume that
 $q_1\geq p_1$, we get the only restriction $0\leq t\leq p_1$.

Taking $\lambda_{11}= p_1$ we get the minimal value for the sum $a_{12} + a_{21}$ which is $p_1 + q_1 - 2 p_1= q_1-p_1$.

This shows that  $\bar{d}(\mu,\nu)\geq\,q_1\,-\,p_1.$

\bigskip

Now we will show that there exists a plan $\lambda$ that realizes the value $q_1-p_1$.

Consider a new Bernoulli independent process $(X_n,Y_n)$ with value on $\{1,2\} \times \{1,2\}.$

We set  $$P( X_0=1,Y_0=1) = p_1,\,P( X_0=1,Y_0=2) =\,\,0,$$
$$\,P( X_0=2,Y_0=1) = q_1-p_1,\,P( X_0=2,Y_0=2) =q_2=p_2-q_1 +p_1 ,\,
$$

It is easy to se that such plan $\lambda$ is in $J(\mu,\nu)$ and  $\int \, I_{\{x_0\neq y_0\}}\lambda (dx, dy)=q_1-p_1$.\qed

A natural question is what can be said for the $\bar{d}$ distance of two finite state Markov Processes (taking values in the same state space  $S$) obtained from two different stochastic matrices, or more generally, for Gibbs states. This is not a so easy problem and partial results can be found for instance, in \cite{Ellis2} \cite{GLT}  \cite{BFG} and \cite{BFG2}.
In most of them, couplings are used in an essential way.

We refer the reader to \cite{FG} \cite{Perez} \cite{To} for more details on couplings from a probabilistic point of view.

%\end{document}

\section{Contraction for the dual of the Ruelle Operator} \label{Ru}

In this section, we study properties related to Gibbs states of
Lipchitz potentials (see \cite{PP} or \cite{LFT} for general results on the topic).

Suppose $A=\log J$ is Lipchitz and normalized, that is, for any $x \in \Omega$ we have $\mathcal{L}_{\log J} (1) (x) =1.$
\bigskip

 If  $ x=(x_1,x_2,..)\in \{1,2,...,d\}^\mathbb{N}$ and $t\in \mathbb{N},$ we denote by  $z_j^t (x)$, $j=1,2,...,d^t$, the
 $d^t$ solutions of $\sigma^t(z)=x.$

  We denote $J^t (z^t_j(x))$ the expression $e^{\sum_{k=0}^{t-1} \log J(\sigma^k (z^t_j(x))  }$.

Therefore, given $x$ and $y$ we have two probabilites
$$(P^t)^*(\delta_x)=\sum_{j=1}^{d^t} \delta_{z_j^t(x)} J^t (z_j^t(x))$$ and
$$(P^t)^* (\delta_y)= \sum_{j=1}^{d^t} \delta_{z_j^t(y)} J^t (z_j^t(y)).$$

For each $k$ we have that the points $z_j^k(x)$ and $z_j^k(y)$, $j=1,2,3,..,2^k$, are all different but the distance $d_\theta( z_j^k(x),z_j^k(y))$ is at most $\theta^k$ ($j$ pair by pair).

Suppose $A=\log J$ has Lipchitz constant $M$.
Then, for $t$ and $j=1,2,3,..,2^t$ fixed
$$|\, \sum_{k=0}^{t-1} \log J(\sigma^k (z^t_j(x))- \sum_{k=0}^{t-1} \log J(\sigma^k (z^t_j(y))\,|\leq$$
$$ \sum_{k=0}^{t-1} M\, d_\theta(\sigma^k (z^t_j(x))- \sigma^k (z^t_j(y))\,)\,|\leq\,$$
$$M\,[\,   \sum_{k=0}^{t-1} \theta^k\,]\,d_\theta(x,y)\leq M \frac{1}{1-\theta} \, d_\theta(x,y)$$

In this way for any $x$ and $y$ we have for any $t$, $j=1,2,3,..,2^{\,t}$,
\begin{equation}\label{bounded_distortion} \frac{J^t (z^t_j(x))}{J^t (z^t_j(y))} \,\leq e^{M \frac{1}{1-\theta} \, d_\theta(x,y) } .\end{equation}

The above kind of estimation is known as the bounded distortion property. It is a key element in the subsequent developments.

\medskip

\begin{lemma} For every $\delta>0$, there is $T$ such that for any $k\geq T$ there exists an $a>0$, so that
$$ \sup_{\Gamma\in \mathcal{C} (\,(P^k)^*(\delta_x),\, (P^k)^*(\delta_y))}\, \Gamma \,\{(x',y'))\in\, \Omega\times \Omega \,:\, d_\theta (x',y') \leq  \delta   \,\}\,\geq a.$$
\end{lemma}

%\end{document}

{\bf Proof.}

In order to prove the Lemma, given $x$ and $y$ we construct explicitly an element in
$$\mathcal{C} (\,(P^k)^*(\delta_x),\, (P^k)^*(\delta_y)).$$

By the orbit structure of the Bernoulli shift, the preimages of $x$ and $y$ come in pairs, and as stated above, the distance of a pair satisfies  $d_\theta( z_j^k(x),z_j^k(y)) < \theta^k$. Hence, if $T > \log \delta / \log \theta$, then $d_\theta( z_j^T(x),z_j^T(y)) < \delta$. The other basic observation for the construction stems from bounded distortion in (\ref{bounded_distortion}). Namely, there exists $\Lambda \in (0,1]$, independent from $T$, $j$ and $x$ such that\footnote{$\Lambda=1$ if and only if $J$ is constant on cylinders of length 1} for $j=1,2,3,..,d^T$
$$ \Lambda J^T(z_j^T(x))  \leq  \alpha_{j}^T:= \inf_{z \in \Omega}{J^T(z_j^T(z))} \leq J^T(z_j^T(x)) .   $$
For $\beta_{j}^T(x) :=J^T(z_j^T(x))- \alpha_{j}^T  $, $j=1,2,3,..,d^T$, we hence obtained a decomposition into a strictly positive part $\beta_{j}^T(x)$ and a
 strictly positive part $\alpha_{j}^T$, which is independent from $x$ (and $y$) but comparable to  $J^T(z_j^T(x))$ by $\Lambda$.
Hence, we obtain a decomposition of a probability measure into two sub-probability measures by
 $$(P^T)^*(\delta_x)=\sum_{j=1}^{d^T} \delta_{z_j^T(x)}  \alpha_{j}^T +        \sum_{j=1}^{d^T} \delta_{z_j^T(x)} \beta_{j}^T(x) =: \mu_T +  \nu_T^x = $$
 $$\sum_{j=1}^{d^T} \delta_{z_j^T(x)} J^T (z_j^T(x))  .$$

We claim that the probability $\Gamma$ on $\Omega \times \Omega$ defined by
\[ \Gamma :=  \sum_{j=1}^{d^T} \delta_{(z_j^T(x),z_j^T(y))}  \alpha_{j}^T +  \frac{1}{ \nu_T^x(\Omega)} \nu_T^x \otimes \nu_T^y.\]
is in $\Gamma \in \mathcal{C} (\,(P^k)^*(\delta_x),\, (P^k)^*(\delta_y))$.

Indeed, consider a Borel set $A$, then
$$\Gamma (\Omega \times A) = \sum_{j=1}^{d^T} \delta_{(z_j^T(x),z_j^T(y))}  \alpha_{j}^T (\Omega \times A)\, +  \,\frac{1}{ \nu_T^x(\Omega)} \nu_T^x \otimes \nu_T^y\, (\Omega \times A)=$$
$$ \sum_{j=1}^{d^T} \delta_{z_j^T(y)}  \alpha_{j}^T (A)\,+\,  \nu_T^y(A)=  \mu_T (A) +  \nu_T^y(A)=$$
$$  \sum_{j=1}^{d^T} \delta_{z_j^T(y)} J^T (z_j^T(y))\,(A)\,=\, (P^T )^* (\delta_y)\,(A).$$

Note that $\nu_T^x(\Omega)= \nu_T^y(\Omega)= 1 - \mu_T (\Omega).$

In the same way, given $A$ we have
$$\Gamma (A \times \Omega) = \sum_{j=1}^{d^T} \delta_{(z_j^T(x),z_j^T(y))}  \alpha_{j}^T (A \times \Omega)\, +  \,\frac{1}{ \nu_T^x(\Omega)} \nu_T^x \otimes \nu_T^y\, (A \times \Omega)=$$
$$ \sum_{j=1}^{d^T} \delta_{z_j^T(x)}  \alpha_{j}^T (A)\,+\,  \nu_T^x(A)\, \frac{\nu_T^y(\Omega)}{\nu_T^x (\Omega)} =  \mu_T (A) +  \nu_T^x(A)=$$
$$  \sum_{j=1}^{d^T} \delta_{z_j^T(x)} J^T (z_j^T(x))\,(A)\,=\, (P^T )^* (\delta_x)\,(A).$$

The analogous result for sets of the form $\Omega\times A$ is true.

This shows that $\Gamma \in \mathcal{C} (\,(P^k)^*(\delta_x),\, (P^k)^*(\delta_y))$.

We claim that $\Gamma$ also satisfies
\[  \Gamma \,\{(x',y'))\in\, \Omega\times \Omega \,:\, d_\theta (x',y') \leq  \delta   \,\} \geq  \Lambda.   \]

Indeed,

$$  \Gamma \,\{(x',y'))\in\, \Omega\times \Omega \,\, d_\theta (x',y') \leq  \delta   \,\} \geq$$
 $$\sum_{j=1}^{d^T} \delta_{(z_j^T(x),z_j^T(y))} \,\, \alpha_{j}^T \,[\,\,\{(x',y'))\in\, \Omega\times \Omega \,:\, d_\theta (x',y') \leq  \delta   \,\}\,\,]=$$
 $$ \sum_{j=1}^{d^T} \delta_{z_j^T(x)} \,\, \alpha_{j}^T  (x)\geq  \sum_{j=1}^{d^T}    \Lambda J^T(z_j^T(x))    \,=\,\Lambda    $$

\medskip

This proves the Lemma for $a := \Lambda$.\qed

\bigskip

We define
$$ \text{var}_n f = \sup \{ |f(x)-f(y)|\,: x_i=y_i,\, 0\leq i <n \},$$
and the pseudo norm
$$|f|_\theta= \sup \{\frac{\text{var}_n f}{\theta^n}\,\,:\, n \geq 0\}.$$

\bigskip

We know that if $\log J$ is Lipchitz, then, there exist $C>0$ and $\alpha_1\in(0,1)$ such that for any $t\geq 0$, and any Lipchitz function
$\phi$ we have (prop 2.1 in \cite{PP})

$$|\mathcal{L}^t (\phi)\,|_\theta \leq  C\, \sup_x  |\phi(x)| + \theta^t |\phi|_\theta. $$

We wrote this in the notation of \cite{Hair} as

$$ \sup \{\frac{\text{var}_n \,\mathcal{L}^t (\phi) }{\theta^n}\,\,:\, n \geq 0\} \leq \, C\, \sup_x  |\phi(x)|\, +\,\, \alpha_1  \sup \{\frac{\text{var}_n \,\phi}{\theta^n}\,\,:\, n \geq 0\} . $$

\bigskip

The above expression is known as the Lasota-Yorke inequality.

\medskip

Take $\delta < \frac{1-\alpha_1}{2\, C} $. Now, we define a metric $d(x,y)= \min \{1,\delta^{-1} d_\theta(x,y)\}.$
The two metrics $d$ and $d_\theta$  are equivalent.

Remember that we denote $\mathcal{C} (\mu_1,\mu_2) $ the set of plans in $\Omega\times \Omega$ such that the projection in the first coordinate is $\mu_1$ and in the second is $\mu_2$.

Remember also that we define the $1$-Wasserstein metric associated to $d$
$$ d_1( \mu_1,\mu_2)\, = \inf \{ \int\,\int d(x,y) \,\,d \,\,\Gamma(dx,dy)\,\,|\,\, \Gamma \in \mathcal{C} (\mu_1,\mu_2)\,\}.$$

\begin{proposition} There exist $\alpha<1$, where $\alpha =\max \{1- \frac{a}{2}, \frac{1}{2} (1 + \alpha_1)  \},$ and $t>0$,
such that, for any $x,y$
$$ d_1( (P^t)^*(\delta_x), (P^t)^*(\delta_y))\leq \alpha\, d (x,y).$$
\end{proposition}

\bigskip

The proof will be done later.
\medskip

Suppose that this result is proved, then we get:

\begin{theorem} \label{maincon} There exist $\alpha<1$, where $\alpha =\max \{1- \frac{a}{2}, \frac{1}{2} (1 + \alpha_1)  \},$ and $t>0$,
such that, for any $\mu_1,\mu_2$
$$ d_1( (P^t)^*(\mu_1), (P^t)^*(\mu_2))\leq \alpha\, d_1 (\mu_1,\mu_2).$$
\end{theorem}

{\bf Proof:} The main idea is to prove first the following claim: suppose  $Q$ is the $d_1$-optimal plan for $\mu_1$ and $\mu_2$, then,
$$ d_1\,((P^t)^* (\mu_1),(P^t)^* (\mu_2)) \leq \int \,d_1( (P^t)^*(\delta_x), (P^t)^* (\delta_y))\, d Q(dx, dy).$$

Suppose the plan in $\Omega\times \Omega$ denoted by $Q(dx, dy)$ has marginals $\mu_1$ and $\mu_2$ in respectively the first and second coordinates.

We will prove the result for a more general continuous potential $c$. Then, you just have to take $c=d$ in order to get the claim.

Given a continuous cost $c(z_1,z_2)$, $c:X \times X \to \mathbb{R}$,  we assume that $Q$ is $c$-optimal for $\mu_1$ and $\mu_2$. Now, given two points $x,y$ suppose
$R^{x,y} (dz_1,dz_2)$ is the $c$-optimal probability plan for $P^* (\delta_{x})$ and  $P^* (\delta_{y})$.

We denote $ S(d\,z_1,d\,z_2)$ the plan
$$S(dz_1,dz_2)\,=\,\int \int     R^{x,y} (dz_1,dz_2)\, Q( dx, dy).$$

We are going to show that the marginals of this plan are $\mu_1$ and $\mu_2$.

Indeed,

$$ \int \int\, \varphi(z_1) S( dz_1,dz_2)= $$
$$\int\int\int\int  \,\varphi(z_1) R^{ x,y} ( d z_1,d z_2)\, Q ( dx , dy) \,dx\, dy\, dz_1\,dz_2=$$
$$\int\int\int  \,\varphi(z_1) P^* (\delta_x)\, Q ( dx , dy) \,dx dy dz_1=$$
$$\int\,[\, P^*\,(\int\int  \,\varphi(z_1) \, Q (\, dx\, , dy)dy dz_1\,) \, ]\, (\delta_x) = $$
$$\int\,[ \,P^*\,(\int \,\varphi(.) \, Q (\, .\, , dy)dy \,) \, (\delta_x)\,]\, dx= \int\,\, \varphi (x) \,\,d P^* (\mu_1)\,\,( dx)   $$
because $P^*$ is linear on measures.

In this way the first marginal of $S( dz_1,dz_2)$ is $P^* \,(\mu_1)$.

In the same way one can prove that the second marginal  of $S( dz_1,dz_2)$ is $P^*\,\mu_2$.
 Now we consider $c(x,y)=d(x,y).$

From the above, we get
$$ d_1 (P^* \,(\mu_1), P^* \,(\mu_2))   \leq \int d(x,y)\, S(dx,dy),$$
where $S$ was defined from $Q$ which is the $d$-optimal plan for $\mu_1$ and $\mu_2$.
\bigskip

Therefore, from the above
$$ d_1\,(P^* (\mu_1),P^* (\mu_2)) \leq \int \,d_1( P^*(\delta_x), P^* (\delta_y))\, d Q(dx, dy),$$
where $Q$ is the $d_1$-optimal plan for $\mu_1$ and $\mu_2$.

In a similar way, given $t>0$ on can show the analogous result
$$ d_1\,((P^t)^* (\mu_1),(P^t)^* (\mu_2)) \leq \int \,d_1( (P^t)^*(\delta_x), (P^t)^* (\delta_y))\, d Q(dx, dy),$$
where $Q$ is the $d_1$-optimal plan for $\mu_1$ and $\mu_2$.

Therefore,
if  there exists a $t>0$ and  $\alpha<1$, such that, for any $x$ and $y$ we have
$$d_1( (P^t)^*(\delta_x) ,(P^t)^* (\delta_y)) \,\leq \,\alpha\, d(x,y),$$
then
$$ d_1((P^t)^* (\mu_1),(P^t)^* (\mu_2) \,\leq\, \alpha \,\,d_1(\mu_1,\mu_2).$$

This is so because
$$ d_1\,((P^t)^* (\mu_1),(P^t)^* (\mu_2)) \leq \int \,d_1( (P^t)^*(\delta_x), (P^t)^* (\delta_y))\, d Q(dx, dy)\leq$$
 $$\alpha \, \int \,d(x,y)\, d Q(dx, dy)=  \alpha \,\,d_1(\mu_1,\mu_2).$$\qed

\bigskip

Now we prove that
for any $x,y$
$$ d_1( (P^t)^*(\delta_x), (P^t)^*(\delta_y))\leq \alpha\, d (x,y).$$

Suppose $d_\theta(x,y)\leq \delta,$ where  $\delta< \frac{1-\alpha_1}{2\, C} $.

Remember that $d(x,y)= \min \{1,\delta^{-1} d_\theta(x,y)\}.$ In this case $d(x,y)= \delta^{-1} d_\theta(x,y).$

By Kantorovich duality  (see \cite{Vi1} and \cite{Vi2})
$$ d_1( \mu_1,\mu_2)\, = \sup_{\phi:X\to \mathbb{R}\,\,\text{has $d$\, Lipchitz constant}\, \leq 1} \,\{\, \int \phi \, d\,\mu_1 - \int \phi\, d\mu_2\,\}.$$

We have to show that: if $\phi$ has $d$-Lipchitz constant smaller than $1$, then, for such pair of $x,y$

$$|\mathcal{L}^t_{\log J} \phi(x) - \mathcal{L}^t_{\log J} \phi(y)|\leq \alpha d(x,y) = \alpha \delta^{-1} \, d_\theta (x,y).$$

We can assume without loss of generality that
$\phi$ attains the value $0$.

In this case $\sup_x  |\phi(x)|\leq 1$.

Moreover,
$$\sup \frac{\delta\,\,| \phi(x) - \phi(y)|  }{d_\theta(x,y)}\leq \sup \frac{| \phi(x) - \phi(y)|  }{d(x,y)}\leq 1. $$

Then,

$$  \,\frac{|\mathcal{L}^t_{\log J} (\phi) (x)- \mathcal{L}^t_{\log J} (\phi) (y)|}{d_\theta(x,y)} \leq $$
$$\, C\, \sup_x  |\phi(x)| + \alpha_1 \sup \frac{| \phi(x) - \phi(y)|  }{d_\theta(x,y)}\leq C + \alpha_1\, \delta^{-1} . $$

As $\delta< \frac{1-\alpha_1}{2\, C} $, then $C<\, \frac{1-\alpha_1}{2} \,\delta^{-1}.$

Therefore, from the above, we get that for any $x,y$ such that $d_\theta(x,y)\leq \delta,$ we have
$$ |\mathcal{L}^t_{\log J} (\phi) (x)- \mathcal{L}^t_{\log J} (\phi) (y)|\leq d_\theta(x,y)\,  (C + \alpha_1\, \delta^{-1})\leq$$
$$  d_\theta(x,y)\,  (\frac{1-\alpha_1}{2} \,\delta^{-1} + \alpha_1\,\delta^{-1} )=$$
$$d_\theta(x,y)\, \delta^{-1} \,( \frac{1 +\alpha_1}{2} )= d(x,y)\, \,( \frac{1 +\alpha_1}{2} )\leq d(x,y)\, \alpha,$$
because $\alpha =\max \{1- \frac{a}{2}, \frac{1}{2} (1 + \alpha_1)  \}.$
\bigskip

Now we suppose that
$x,y$ such that $d_\theta(x,y)> \delta.$ This implies that $d(x,y)=1$.

We denote
$$\Delta_\delta=\{(x',y')\in\, \Omega\times \Omega \,:\, d_\theta (x',y') \leq  \frac{1}{2} \delta   \,\}.$$

For such $\delta$ there exists $a>0$ and $T>0$, such that, for $k>T$, there exists a plan $\Gamma=\Gamma_{k}$ which satisfies
$\Gamma\in \mathcal{C} (\,(P^k)^*(\delta_x),\, (P^k)^*(\delta_y))$ and
$  \Gamma \,(\Delta_\delta)\,\geq a.$

Note that if $d_\theta (x',y') \leq  \frac{1}{2} \delta  $, then $d(x',y')\leq \frac{1}{2}.$

Therefore,
$$ \int d(x',y') \, \Gamma (dx',dy')\,\,\leq \frac{1}{2}\Gamma \,(\Delta_\delta)+ 1 - \Gamma \,(\Delta_\delta)= 1 -\frac{1}{2}\Gamma \,(\Delta_\delta)\leq 1 - \frac{a}{2},   $$
because $d(x',y')\leq 1$ in the complement of $\Delta_\delta$.

Then, if $d_\theta(x,y)> \delta$ we get
$$ d_1( (P^t)^*(\delta_x), (P^t)^*(\delta_y))\leq  \int d(x',y') \, \Gamma (dx',dy')\,\,\leq $$
$$(1 - \frac{a}{2}) = (1 - \frac{a}{2})\, \,d\,(x,y)\,\leq\,\alpha\,\, d(x,y)$$
because $\alpha =\max \{1- \frac{a}{2}, \frac{1}{2} (1 + \alpha_1)  \}.$

From all this, it follows the main result.

\end{document}